\documentclass[11pt]{amsart}

\usepackage[a4paper,hmargin=3cm,vmargin=3cm]{geometry}
\usepackage{amsfonts,amssymb,amscd,amstext}
\usepackage{graphicx}
\usepackage[dvips]{epsfig}
\usepackage[latin1]{inputenc}

\usepackage{fancyhdr}
\usepackage{times}

\usepackage{enumerate}
\usepackage{titlesec}
\usepackage{mathrsfs}
\usepackage{stmaryrd}

\def\R{\mathbb{R}}

\def\Z{\mathbb{Z}}
\def\C{\mathbb{C}}

\def\D{\mathbb{D}}

\newcommand{\ben}{\begin{enumerate}}
\newcommand{\bit}{\begin{itemize}}
\newcommand{\een}{\end{enumerate}}
\newcommand{\eit}{\end{itemize}}

\newcommand{\ed}{\end{document}}

\newcommand{\cal}{\mathcal}

\def\cU{\mathcal{U}}

\def\cS{\mathcal{S}}

\def\cR{\mathcal{R}}
\def\cB{\mathcal{B}}
\def\cW{\mathcal{W}}
\def\cS{\mathcal{S}}
\def\cV{\mathcal{V}}

\def\cH{\mathcal{H}}
\def\cL{\mathcal{L}}

\def\cG{\mathcal{G}}

\def\cN{\mathcal{N}}

\let\8=\infty \let\0=\emptyset  
\let\hat=\widehat

\let\tilde=\widetilde

\let\landa=\lambda
\let\alfa=\alpha

\let\parc=\partial

\def\ep{\varepsilon}

\def\landa{\lambda}

\def\flecha{\rightarrow}
\def\esiz{\langle}
\def\esde{\rangle}

\newcommand{\fl}{\longrightarrow}
\def\r{\mathbb{R}}
\def\S{\Sigma}

\def\cte.{\mathop{\rm cte.}\nolimits}

\def\R{\mathbb{R}}
\def\Z{\mathbb{Z}}
\def\C{\mathbb{C}}

\def\D{\mathbb{D}}

\def\H{\mathbb{H}}
\def\S{\mathbb{S}}

\newfont{\bb}{msbm10 at 12pt}

\def\r{\mathbb{R}}

\titleformat{\section}
{\filcenter\bfseries\large} {\thesection{.}}{0.2cm}{}
\titleformat{\subsection}[runin]
{\bfseries} {\thesubsection{.}}{0.15cm}{}[.]
\titleformat{\subsubsection}[runin]
{\em}{\thesubsubsection{.}}{0.15cm}{}[.]

\usepackage[up,bf]{caption}

\newtheorem{theorem}{Theorem}[section]
\newtheorem{lemma}[theorem]{Lemma}

\newtheorem{remark}[theorem]{Remark}
\newtheorem{corollary}[theorem]{Corollary}
\newtheorem{definition}[theorem]{Definition}

\newtheorem{assertion}[theorem]{Assertion}

\theoremstyle{definition}

\numberwithin{equation}{section}
\numberwithin{figure}{section}

\usepackage{color}

\begin{document}
\fancyhead[LO]{Elliptic Weingarten surfaces}
\fancyhead[RE]{Isabel Fernández, José A. Gálvez, Pablo Mira}
\fancyhead[RO,LE]{\thepage}

\thispagestyle{empty}


\begin{center}
{\bf \LARGE Quasiconformal Gauss maps and the \\[0.2cm] Bernstein problem for Weingarten multigraphs}
\vspace*{5mm}

\hspace{0.2cm} {\Large Isabel Fernández, José A. Gálvez, Pablo Mira}
\end{center}

\footnote[0]{
\noindent \emph{Mathematics Subject Classification}: 53A10, 53C42, 35J15, 35J60. \\ \mbox{} \hspace{0.25cm} \emph{Keywords}: Weingarten surfaces, fully nonlinear elliptic equations, Bernstein problem, multigraphs, curvature estimates, quasiconformal Gauss map.}



\vspace*{7mm}

\begin{quote}
{\small
\noindent {\bf Abstract}\hspace*{0.1cm}
We prove that any complete, uniformly elliptic Weingarten surface in Euclidean $3$-space whose Gauss map image omits an open hemisphere is a cylinder or a plane. This generalizes a classical theorem by Hoffman, Osserman and Schoen for constant mean curvature surfaces. In particular, this proves that planes are the only complete, uniformly elliptic Weingarten multigraphs. We also show that this result holds for a large class of non-uniformly elliptic Weingarten equations. In particular, this solves in the affirmative the Bernstein problem for entire graphs for that class of elliptic equations. To obtain these results, we prove that planes are the only complete multigraphs with quasiconformal Gauss map and bounded second fundamental form.
\vspace*{0.1cm}

}
\end{quote}


\section{Introduction}

A \emph{Weingarten surface} is an immersed surface $\Sigma$ in $\R^3$ whose mean curvature $H$ and Gauss curvature $K$ are related by some smooth equation
\begin{equation}\label{wein1}
W(H,K)=0.
\end{equation}
In this paper, we will require that $W\in C^2(\R^2)$. We say that $\Sigma$ is an \emph{elliptic} Weingarten surface if \eqref{wein1} is elliptic when viewed as a fully nonlinear second order PDE in local graphical coordinates on $\Sigma$. In the elliptic case, \eqref{wein1} can be rewritten as
 \begin{equation}\label{weq2}
 H=g(H^2-K), \hspace{1cm}  4t (g'(t))^2<1 \hspace{0.5cm} \text{ for all } t\geq 0,
 \end{equation}
for some $C^2$ function $g:[0,\8)\flecha \R$; the inequality in \eqref{weq2} is precisely the ellipticity condition for the equation. Note that, when $g$ is constant, \eqref{weq2} is the constant mean curvature (CMC) equation. Elliptic Weingarten surfaces are often called \emph{special Weingarten surfaces}.  Their global geometry has been studied in depth by many authors; see e.g. \cite{AEG,A,B,Ch,CF,EM,GMM,GM3,HW1,Ho0,Ho,RS,ST,ST2,T}.

The most fundamental open problem in the global theory of elliptic Weingarten surfaces is probably the \emph{Bernstein problem}, see e.g. Rosenberg and Sa Earp \cite{RS}:

\vspace{0.2cm}

\noindent {\bf Bernstein problem:} \emph{Are planes the only entire elliptic Weingarten graphs in $\R^3$?}
 
 \vspace{0.2cm}
 
If $g(0)\neq 0$, there are no entire graphs satisfying \eqref{weq2}, as follows from an easy application of the maximum principle and the fact that spheres of radius $1/|g(0)|$ satisfy \eqref{weq2}. That is, the Bernstein problem is only meaningful for Weingarten surfaces \emph{of minimal type}, i.e., when $g(0)=0$.
 
As noted in \cite{ST2}, the Bernstein problem has an affirmative answer when \eqref{weq2} is \emph{uniformly elliptic}, that is, when there exists some constant $\Lambda\in (0,1)$ such that
 \begin{equation}\label{unife2}
 4t (g'(t))^2\leq \Lambda<1 \hspace{0.5cm} \text{ for all } t\geq 0.
 \end{equation}
 More specifically, the uniform ellipticity condition \eqref{unife2} together with $g(0)=0$ implies that the principal curvatures $\kappa_1, \kappa_2$ of $\Sigma$ satisfy the inequality $$\kappa_1^2 + \kappa_2^2 \leq 2\gamma \, \kappa_1 \kappa_2,$$ for some $\gamma \in \R$; see e.g. Lemma \ref{lem:cuw1}. This inequality is equivalent to the property that $\Sigma$ has \emph{quasiconformal Gauss map}, see Section \ref{sec:cuasi} for the details. A deep theorem by L. Simon (\cite[Theorem 4.1]{Sim}) shows that planes are the only entire graphs with quasiconformal Gauss map. So, in particular, planes are the only entire, uniformly elliptic Weingarten graphs. Note that this result includes the classical Bernstein theorem for minimal surfaces ($H=0$). 
 
 Not much is known about classes of Weingarten surfaces for which the Bernstein problem can be solved, if \eqref{unife2} does not hold (see \cite{RaS}). One of our contributions in this paper is to solve in the affirmative the Bernstein problem for a wide class of non-uniformly elliptic Weingarten equations; see the Corollary at the end of the introduction.

The Bernstein problem is related to the spherical image of the Gauss map $N:\Sigma\flecha \S^2$ of elliptic Weingarten surfaces $\Sigma$ in $\R^3$. Indeed, note that if $\Sigma$ is a graph, $N(\Sigma)$ lies in an open hemisphere. Conversely, if $N(\Sigma)$ lies in some open hemisphere, then $\Sigma$ might not be a graph, but it is a \emph{multigraph}, i.e., a local graph with respect to a specific fixed direction of $\R^3$.

A classical theorem by Hoffman, Osserman and Schoen \cite{HOS} proves that if the Gauss map image $N(\Sigma)$ of a complete CMC surface $\Sigma$ lies in a closed hemisphere, then $\Sigma$ is a plane ($H=0$) or a cylinder ($H\neq 0$). So, this theorem can be seen as a solution to a \emph{generalized Bernstein problem} for CMC multigraphs, and motivates the following problem, see Question 2 in p. 699 of \cite{ST}.

\vspace{0.2cm}

\noindent {\bf Bernstein problem for multigraphs:} \emph{Are planes and cylinders the only complete, elliptic Weingarten surfaces in $\R^3$ whose Gauss map image lies in a closed hemisphere of $\S^2$?}

\vspace{0.2cm}

Observe that this problem asks, in particular, if complete (not necessarily entire) elliptic Weingarten graphs in $\R^3$ are planes. This time, in contrast with the case of entire graphs, the problem is non-trivial if $g(0)\neq 0$ in \eqref{weq2}. Also, one should note that there exist complete, rotational CMC \emph{unduloids} in $\R^3$ whose Gauss map image lies in an arbitrarily small tubular neighborhood of a geodesic of $\S^2$. These examples show the necessity of the hypothesis on the Gauss map image in this problem.

We now state the main results of this paper. In Section \ref{sec:2} we will discuss some preliminary material on Weingarten surfaces and quasiconformal Gauss maps, and among other results we will show (see Lemma \ref{numulti}):

\vspace{0.2cm}

\noindent {\bf Lemma A:} \emph{If the Gauss map image $N(\Sigma)$ of an elliptic Weingarten surface $\Sigma$ lies in a closed hemisphere, then either $\Sigma$ is a multigraph (i.e., $N(\Sigma)$ lies in the interior of this hemisphere), or $\Sigma$ is a piece of a plane or a cylinder}.

\vspace{0.2cm}

Thus, in order to classify elliptic Weingarten surfaces whose Gauss map image lies in a closed hemisphere, it suffices to classify elliptic Weingarten multigraphs.

In Section \ref{sec:3n} we study multigraphs with quasiconformal Gauss map, not necessarily Weingarten surfaces, proving (Theorem \ref{alfa}):

\vspace{0.2cm}

\noindent {\bf Theorem A:} \emph{Planes are the only complete multigraphs with quasiconformal Gauss map and bounded second fundamental form.}

\vspace{0.2cm}

It is a long-standing open problem to determine whether planes are the only complete surfaces in $\R^3$ with quasiconformal Gauss map, and whose Gauss map image $N(\Sigma)$ omits an open set of $\S^2$; see Section 5 in \cite{Sim}. Theorem A can be seen as a step in this direction.

In Section \ref{sec:3} we will use Theorem A to prove that the Bernstein problem for elliptic Weingarten multigraphs (and in particular for entire graphs) can be solved whenever we have a bound on the norm of the second fundamental form. From Lemma A and Theorem \ref{b2} we have:

\vspace{0.2cm}

\noindent {\bf Theorem B:} \emph{Planes and cylinders are the only complete elliptic Weingarten surfaces in $\R^3$ with bounded second fundamental form and Gauss map image contained in a closed hemisphere}.

\vspace{0.2cm}

The proofs of Theorems A and B are inspired by an argument of Hauswirth, Rosenberg and Spruck \cite{HRS} in the context of CMC surfaces in the product space $\H^2\times \R$, where $\H^2$ denotes the hyperbolic plane, and subsequent modifications of it in other geometric theories by Espinar and Rosenberg \cite{ER}, and Gálvez, Mira and Tassi \cite{GMT}; see also Manzano-Rodríguez \cite{Manzano} and Daniel-Hauswirth \cite{DaHa}.

Theorem B reduces the Bernstein problem for elliptic Weingarten graphs or multigraphs to the obtention of a priori estimates for the norm of the second fundamental form (usually called \emph{curvature estimates}). In Section \ref{sec:4} we will prove such a curvature estimate for the \emph{uniformly elliptic} case (Theorem \ref{curves2}). This estimate is based on a paper by Rosenberg, Souam and Toubiana \cite{RST} on stable CMC surfaces in Riemannian $3$-manifolds. A key difficulty in our situation is that, in the natural blow-up process that one uses to obtain such curvature estimates, the equation \eqref{weq2} is lost in the limit. That is, even if one finds a limit surface after the blow-up, this limit surface might not be, in general, an elliptic Weingarten surface anymore. Theorem A, where no elliptic equation appears, will allow us to control this limit surface.
As a consequence of this estimate, we obtain from Lemma A and Theorem \ref{unifeth}:

\vspace{0.2cm}

\noindent {\bf Theorem C:}  \emph{Planes and cylinders are the only complete, uniformly elliptic Weingarten surfaces in $\R^3$ whose Gauss map image is contained in a closed hemisphere}.

\vspace{0.2cm}

Note that Theorem C solves the Bernstein problem for multigraphs in the uniformly elliptic case, and also extends the Hoffman-Osserman-Schoen theorem from CMC surfaces to uniformly elliptic Weingarten surfaces. In order to explain our results in the non-uniformly elliptic case, it is convenient to rewrite the elliptic Weingarten equation \eqref{weq2} in terms of the principal curvatures $\kappa_1,\kappa_2$ of the surface as,
\begin{equation}\label{wein2}
\kappa_2=f(\kappa_1).
\end{equation}
Here, $f$ is a $C^2$ function on an open interval $I_f\subset \R$, that satisfies $f'<0$ (by ellipticity) and $f\circ f = {\rm Id}$ (by symmetry of the relations $2H=\kappa_1+\kappa_2$ and $K=\kappa_1 \kappa_2$). Moreover, $I_f$ is of the form $(a,\8)$, $(-\8,b)$ or $\R$, and the graph $\{(x,f(x)): x\in I_f\}$ is a complete curve in $\R^2$, which is symmetric with respect to the line $y=x$. The uniform ellipticity condition \eqref{unife2} is written for $f$ in \eqref{wein2} as 
\begin{equation}\label{unife}0< \Lambda_1 \leq -f'(x) \leq \Lambda_2 \hspace{0.4cm} \text{for all $x\in I_f=\R$,}
\end{equation} where $\Lambda_1,\Lambda_2$ are positive constants. That is, the slope of the graph of $f$ is negative and uniformly bounded away from $0$ and $-\8$. As a matter of fact, by the symmetry properties of $f$, it suffices to impose one of the two inequalities in \eqref{unife} to obtain the other one.

Note that $I_f=\R$ for any uniformly elliptic Weingarten equation. Thus, the most typical examples of non-uniformly elliptic Weingarten equations happen when the function $f$ in \eqref{wein2} is not globally defined on $\R$; for instance, this is the case of the \emph{linear Weingarten} equation $2aH+bK=c$, with $a,b,c\in \R$ satisfying the ellipticity condition $a^2+bc>0$, and $b\neq 0$; see Section \ref{sec:5}.

When $I_f\neq \R$ and $f(0)=0$, we use Theorem A and the family of parallel surfaces to give an affirmative answer to both the Bernstein problem and the generalized Bernstein problem for multigraphs. By Lemma A and Theorem \ref{bernstein} we have:

\vspace{0.2cm}

\noindent {\bf Theorem D:} \emph{Let $\Sigma$ be a complete elliptic Weingarten surface in $\R^3$ whose Gauss map image is contained in a closed hemisphere. Assume that the function $f$ of its associated Weingarten relation \eqref{wein2} satisfies $f(0)=0$ and is not defined in all $\R$. Then $\Sigma$ is a plane.}

\vspace{0.2cm}

As a matter of fact, Theorem D is a particular case of a much more general result where no Weingarten condition is assumed, but only an inequality between the principal curvatures of the surface; see Theorem \ref{th:erre}. From Theorem D we have:

\vspace{0.2cm}

\noindent {\bf Corollary:} \emph{Planes are the only entire graphs in $\R^3$ that satisfy an elliptic Weingarten equation \eqref{wein2} with $I_f\neq \R$.}

\vspace{0.2cm}

Bernstein's problem remains open when $I_f=\R$ and \eqref{wein2} is not uniformly elliptic. For an alternative formulation of the Corollary above when we write the Weingarten equation as \eqref{weq2} instead of as \eqref{wein2}, see Theorem \ref{bernstein2}.

\section{Elliptic Weingarten surfaces}\label{sec:2}

\subsection{The Weingarten equation}\label{sec:21} 

Let us start by clarifying some aspects about the different ways of writing an elliptic Weingarten equation. 

First, a word of caution. It is not a good idea to work directly with the simple form \eqref{wein1} of the Weingarten equation, because it can be misleading. For instance, both planes and round spheres of radius $1/2$ satisfy the simple linear Weingarten equation $K=2H$, which can be proved to be elliptic. At first sight, this would seem to contradict the maximum principle for elliptic PDEs. This is explained by the fact that the equation $K=2H$ actually contains two different elliptic theories (see Figure \ref{fig:dicur} and the discussion below).
\begin{definition}
An \emph{elliptic Weingarten surface} is an immersed oriented surface $\Sigma$ in $\R^3$ whose mean and Gauss curvatures $H,K$ satisfy a relation \eqref{weq2} for some $g\in C^2([0,\8))$.
\end{definition}
Let us justify this definition next. The Weingarten equation \eqref{wein1} for $W\in C^2(\R^2)$ can be rewritten in terms of the principal curvatures $\kappa_1,\kappa_2$ as \begin{equation}\label{weq}\Phi(\kappa_1,\kappa_2)=0,\end{equation} where $\Phi\in C^2(\R^2)$ is symmetric, i.e. $\Phi(k_1,k_2)=\Phi(k_2,k_1)$. With this formulation, the ellipticity condition for the Weingarten equation is written as (see e.g. \cite{Ho}, p. 129)
 \begin{equation}\label{eliwe}
 \frac{\parc \Phi}{\parc k_1} \frac{\parc \Phi}{\parc k_2}>0 \hspace{1cm} \text{if } \ \Phi=0. \end{equation}
 
Thus, if \eqref{eliwe} holds, we see using the symmetry of $\Phi$ that each connected component of $\Phi^{-1}(0)\subset \R^2$ can be written as a complete graph\begin{equation}\label{weq3}
 k_2=f(k_1), \hspace{1cm} 
 \end{equation}
symmetric with respect to the principal diagonal $k_1=k_2$ in $\R^2$, where $f$ is defined on an interval $I_f:=(a,b)\subset \R$, and satisfies the following conditions:

\begin{enumerate}
\item[(i)]
$f$ is $C^2$, and $f'<0$ (by ellipticity).
 \item[(ii)]
$f\circ f = {\rm Id} $ (by symmetry of $\Phi$).
 \item[(iii)]
If $a\neq -\8$, then $b=+\8$ and $f(x)\to +\8$ as $x\to a$.
 \item[(iv)]
If $b\neq +\8$, then $a=-\8$ and $f(x)\to -\8$ as $x\to b$.
\end{enumerate}

Each connected component of $\Phi^{-1}(0)$ gives rise to a different elliptic theory, with different geometric properties. For instance, the already mentioned Weingarten relation $K=2H$ can be rewritten as $\Phi(\kappa_1,\kappa_2)=\kappa_1+\kappa_2- \kappa_1\kappa_2=0$, and it is clear that $\Phi^{-1}(0)$ has two connected components; see Figure \ref{fig:dicur}. In one of them, all surfaces have principal curvatures greater than $1$, and so are convex, while all surfaces of the other connected component have non-positive curvature.

\begin{figure}[htbp] 
    \centering
    \includegraphics[width=.45\textwidth]{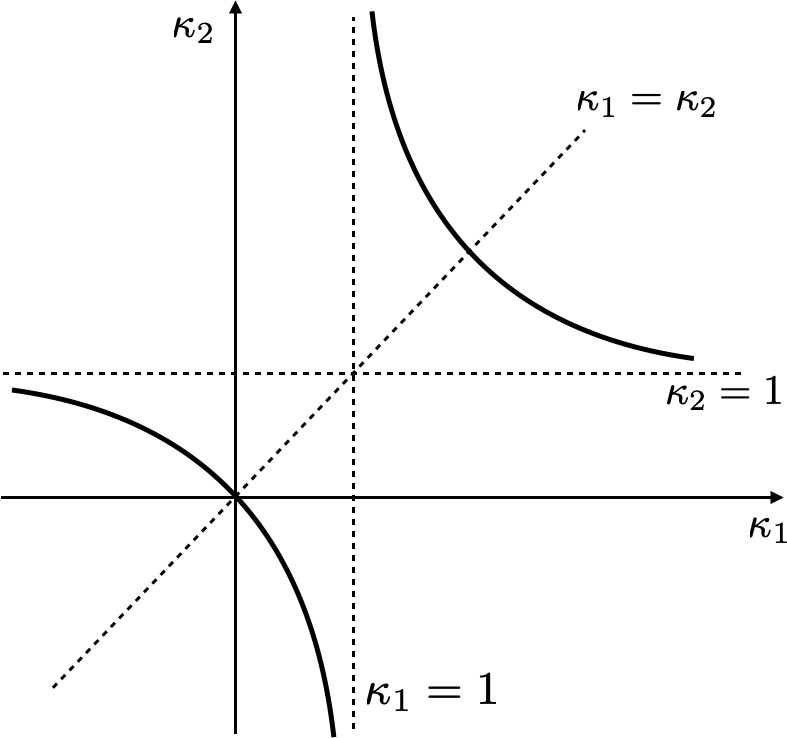}
    \caption{The two connected components of $\Phi^{-1}(0)$ in the $(\kappa_1,\kappa_2)$-plane, for the Weingarten equation $\Phi(\kappa_1,\kappa_2)=0$ corresponding to $K=2H$.}
    \label{fig:dicur}
  \end{figure}

Alternatively, and also by the symmetry and ellipticity conditions on $\Phi$, it is easy to see that each connected component of $\Phi^{-1}(0)$ can be seen as a graph of the form
 \begin{equation}
 \frac{k_1+k_2}{2} = g\left(\frac{(k_1-k_2)^2}{4}\right),
 \end{equation}
where $g\in C^2([0,\8))$ satisfies, by the ellipticity inequality \eqref{eliwe}, the condition
 \begin{equation}\label{elihk}
 4t (g'(t))^2<1 \hspace{1cm} \text{ for all } t\geq 0.
 \end{equation}
This shows that there is no loss of generality in working with \eqref{weq2} or with \eqref{wein2} when dealing with a class of elliptic Weingarten surfaces in $\R^3$, and that both formulations are \emph{essentially} equivalent. The equivalence is not complete because the smoothness of $\Phi$ in \eqref{weq}, or equivalently of $f$ in \eqref{wein2}, does not imply $C^1$-smoothness of $g$ in \eqref{weq2}; see e.g. \cite{HW1}. 

We note that for graphs $z=u(x,y)$, the Weingarten equation \eqref{weq2} is equivalent to the fully nonlinear elliptic PDE
\begin{equation}\label{eedp}F(u_x,u_y,u_{xx},u_{xy},u_{yy})=0,
 \end{equation}
 where $F(p,q,r,s,t):= \cH-g(\cH^2-\mathcal{K})\in C^{2}(\R^5)$, for
 \begin{equation}\label{haka}
 \cH(p,q,r,s,t):=\frac{(1+q^2)r - 2 p q s  + (1+p^2)t}{2(1+p^2+q^2)^{3/2}}, \hspace{0.5cm} \mathcal{K}(p,q,r,s,t):= \frac{rt-s^2}{(1+p^2+q^2)^2}.
 \end{equation}
Here, the ellipticity of \eqref{eedp} is equivalent to the ellipticity condition \eqref{elihk} for $g$. In particular, if $g$ verifies \eqref{elihk}, the class $\cW_g$ of elliptic Weingarten surfaces in $\R^3$ given by \eqref{weq2} satisfies the maximum principle in its usual geometric version.

The number $\alfa:=g(0)$ has a geometric meaning in the class of Weingarten surfaces $\cW_g$, and is called the \emph{umbilical constant} of the class $\cW_g$, because umbilics of any surface in $\cW_g$ have principal curvatures equal to $\alfa$. Note that by making, if necessary, the change $g(t)\mapsto -g(t)$ in \eqref{weq2} while reversing the orientation of the surface, we may assume without loss of generality that $\alfa=g(0)\geq 0$. If $g(0)=0$, planes belong to the Weingarten class $\cW_g$, and the Weingarten equation \eqref{weq2} is said to be \emph{of minimal type}. If $g(0)>0$, spheres of radius $1/\alfa$ with their inner orientation are elements of $\cW_g$. 

When we write the Weingarten equation as \eqref{wein2}, the umbilical constant $\alfa$ is given by the relation $f(\alfa)=\alfa$, and the equation is of minimal type if $f(0)=0$. Moreover, if $f$ is defined at $x=0$ with $f(0)\neq 0$, the cylinders in $\R^3$ of principal curvatures $\{0,f(0)\}$ are elliptic Weingarten surfaces satisfying \eqref{wein2}. 

\begin{definition}
The \emph{curvature diagram} of an immersed oriented surface $\Sigma$ is given by $$\{(\kappa_1(p),\kappa_2(p)) : p \in \Sigma\} \subset \R^2,$$ where $\kappa_1(p)\geq \kappa_2(p)$ are the principal curvatures of $\Sigma$ at $p$. Note that the curvature diagram always lies in the half-plane $x\geq y$ of $\R^2$.
\end{definition}

Let us observe that the curvature diagram of an elliptic Weingarten surface is a subset of a regular curve of $\R^2$ of the form \eqref{weq3}.

\subsection{Quasiconformal Gauss maps and Weingarten surfaces}\label{sec:cuasi}

The fact that a surface in $\R^3$ has quasiconformal Gauss map can be characterized in terms of an inequality for its principal curvatures, see e.g. \cite{Sim} or equation (16.88) in \cite{GT}. We adopt this characterization as a definition here:
 
 \begin{definition}\label{def:cuasic}
An immersed oriented surface $\Sigma$ in $\R^3$ has \emph{quasiconformal Gauss map} if its principal curvatures $\kappa_1\geq\kappa_2$ satisfy at every point of $\Sigma$ the inequality 
 \begin{equation}\label{cuasiconfo}
\kappa_1^2 +\kappa_2^2 \leq 2\gamma\, \kappa_1 \kappa_2,
\end{equation}
for some $\gamma\in \R$.
 \end{definition}
Let us make some comments regarding this definition. If $|\gamma|<1$, the Gauss map $N$ of $\Sigma$ is constant, and thus $\Sigma$ is a piece of a plane. If $\gamma\leq -1$ (resp. $\gamma\geq 1$), then $\Sigma$ has non-positive (resp. non-negative) Gauss curvature, and \eqref{cuasiconfo} is equivalent to the fact that the curvature diagram of $\Sigma$ lies inside a \emph{wedge region} $\cR$ of the half-plane $x\geq y$ limited by two straight lines of $\R^2$ passing through the origin. When $\gamma\leq -1$ these two lines have negative slopes, and so we obtain a region $\cR$ as in Figure \ref{fig:wedges1}, left. When $\gamma\geq 1$, these lines have positive slopes and, since the curvature diagram lies in the half-plane $x\geq y$, we can consider without loss of generality that the region $\cR$ is of the form presented in Figure \ref{fig:wedges1}, right.

\begin{figure}[htbp]
    \centering
    \includegraphics[width=0.9\textwidth]{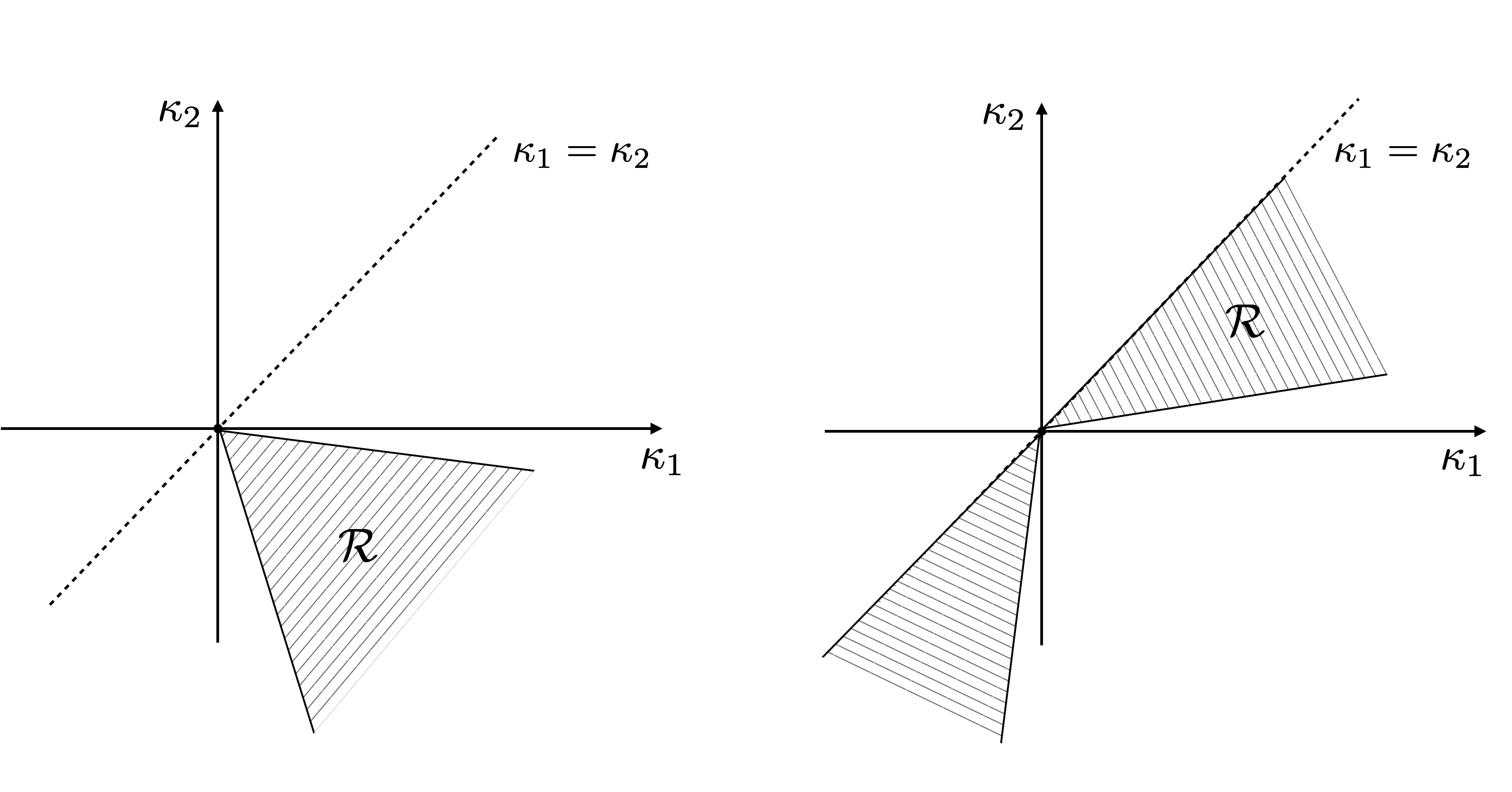}
    \caption{Curvature diagram regions for surfaces with quasiconformal Gauss map, corresponding to the cases $\gamma\leq -1$ (left) and $\gamma\geq 1$ (right) in \eqref{cuasiconfo}.}
    \label{fig:wedges1}
\end{figure}

Let us relate the above concept with the notion of \emph{quasiregular mapping}, i.e., a map $g:\Omega\subset \C\flecha \C$ of class $W_{{\rm loc}}^{1,2}$ that satisfies on $\Omega$ the Beltrami inequality $$|g_{\bar{z}}|\leq \mu |g_z| \hspace{0.5cm} a.e.$$ where $\mu\in [0,1)$ is a constant; see e.g. \cite{AMI}. We remark that quasiregular mappings are sometimes called quasiconformal mappings, as happens in \cite{GT}. By a $K$-quasiregular mapping ($K\geq 1$) we will mean a quasiregular mapping with associated constant $\mu$ given by $\mu=\frac{K-1}{K+1}$.

Let $\Sigma$ be an immersed oriented surface in $\R^3$, with Gauss map $N$. Let $\pi:\S^2\flecha \C\cup \{\8\}$ denote the stereographic projection, and consider the projected Gauss map $g=\pi \circ N:\Sigma\flecha \C\cup \{\8\}$. Consider $z=s+it$ a conformal parameter on $\Sigma$. Then, it is well-known (see \cite{Ken}, p. 90) that $g$ satifies the Beltrami equation
$$H g_z= \phi g_{\bar{z}},$$ where $H$ is the mean curvature of $\Sigma$ and $|\phi |^2 = H^2-K = \frac{1}{4}(\kappa_1-\kappa_2)^2$. Therefore, $$\frac{|g_{\bar{z}}|^2}{|g_z|^2} = \frac{(\kappa_1+\kappa_2)^2}{(\kappa_1-\kappa_2)^2}.$$ From this, it is easily checked that $g$ is quasiregular for some $\mu\in [0,1)$ if and only if \eqref{cuasiconfo} holds for $\gamma=\frac{\mu^2+1}{\mu^2-1}\leq -1.$ Similarly, we have that $g^*(z):=g(\bar{z})$ is quasiregular for $\mu$ if and only if \eqref{cuasiconfo} holds for $\gamma= \frac{1+\mu^2}{1-\mu^2}\geq 1.$

The following convergence result will be used later on.

\begin{lemma}\cite[Corollary 5.5.7]{AMI}\label{lemcuasii}
Suppose $z_0,z_1\in \C$, and let $g_n:\Omega\flecha \C$ be a sequence of $K$-quasiregular mappings defined on an open domain $\Omega\subset \C$, each $g_n$ omitting these two values. Then, there is a subsequence converging locally uniformly on $\Omega$ to a mapping $g$, $$g_{n_k}\flecha g,$$ and $g$ is a $K$-quasiregular mapping (maybe a constant).
\end{lemma}

Also for future reference, we state the next elementary lemmas.

\begin{lemma}\label{opencuasi}
Let $\Sigma$ be an immersed surface with quasiconformal Gauss map $N$. Then either $\Sigma$ is a piece of plane, or $N$ is an open mapping.
\end{lemma}
\begin{proof}
Let $\gamma$ be the constant for which \eqref{cuasiconfo} holds for $\Sigma$, and let $g=\pi \circ N$. If $\gamma\leq -1$, it follows from the above discussion that $g$ is a quasiregular mapping; hence, it is a well-known consequence of the Stoilow factorization theorem (see e.g. \cite{AMI}) that $g$ is an open mapping if $g$ is not constant. So, the result is true for that case. If $\gamma \in (-1,1)$, we know that $N$ is constant. Finally if $\gamma\geq 1$, we know that $g^*(z):=g(\bar{z})$ is quasiregular, and so, again, is either constant or open. The result then follows.
\end{proof}

\begin{lemma}\label{lem:cuw1}
Let $\Sigma$ be a uniformly elliptic Weingarten surface of minimal type. Then, $\Sigma$ has quasiconformal Gauss map, with $\gamma\leq -1$.
\end{lemma}
\begin{proof}
By hypothesis, $\Sigma$ verifies \eqref{wein2} with $f(0)=0$, and $f$ satisfies the uniform ellipticity condition \eqref{unife}. Under these asumptions, it is clear that the curvature diagram $(\kappa_1(\Sigma),\kappa_2(\Sigma))\subset \R^2$, $\kappa_1\geq \kappa_2$, lies on a region $\cR$ of the form, 
\begin{equation}\label{wedge}
 \cR= \{(x,y): x\geq y, m_1x\leq y \leq m_2x\}, \hspace{1cm} m_1,m_2<0.
 \end{equation} 
as in Figure \ref{fig:wedges1}, left. Hence, $\Sigma$ has quasiconformal Gauss map, with $\gamma\leq -1$.
\end{proof}

\begin{lemma}\label{lem:cuw2}
Let $\Sigma$ be an elliptic Weingarten surface of minimal type with bounded second fundamental form. Then, $\Sigma$ has quasiconformal Gauss map, with $\gamma\leq -1$.
\end{lemma}
\begin{proof}
Again, $\Sigma$ satisfies \eqref{wein2} with $f(0)=0$. Since this time $|\sigma|^2:=\kappa_1^2+\kappa_2^2$ is bounded, it is clear by monotonicity of $f$ and the condition $f(0)=0$ that the set $(\kappa_1(\Sigma),\kappa_2(\Sigma))\subset \R^2$ lies again on a wedge region $\cR$ of the form \eqref{wedge}, and so $N$ is quasiconformal with $\gamma\leq -1$.
\end{proof}

One should observe that, for a surface $\Sigma$, the condition of having quasiconformal Gauss map is rather weak, in the sense that \eqref{cuasiconfo} is an inequality, and not an equation between the principal curvatures of $\Sigma$. 

\begin{remark}\label{caposi}
Let $\Sigma$ be a complete surface in $\R^3$ satisfying \eqref{cuasiconfo}, for some $\gamma\geq 1$. Then $\Sigma$ has non-negative Gauss curvature, and so, by Sacksteder's theorem \cite{Sa}, $\Sigma$ is the boundary of a convex body of $\R^3$. In particular, either $\Sigma$ is compact, or the Gauss map image $N(\Sigma)$ lies in a closed hemisphere. But now, by Lemma \ref{opencuasi}, it follows that if $\Sigma$ is not compact, then either $\Sigma$ is a plane or $N(\Sigma)$ lies in an open hemisphere. In the second case, by convexity, $\Sigma$ must be a proper graph over a convex domain. However, since $\Sigma$ has quasiconformal Gauss map, it follows by Simon's theorem that in this situation, $\Sigma$ must actually be a plane (see Theorem 3.1 in \cite{Sim}, and the end of the proof of Theorem \ref{alfa} in this paper).

In particular, any complete, non-compact surface in $\R^3$ satisfying \eqref{cuasiconfo}, for some $\gamma\geq 1$ is a plane.

In contrast, there exist many complete, non-compact surfaces $\Sigma$ that satisfy \eqref{cuasiconfo} for some $\gamma\leq -1$. For instance, any complete minimal surface in $\R^3$ satisfies \eqref{cuasiconfo} for $\gamma=-1$.  \end{remark}

\subsection{Weingarten multigraphs}

Given an immersed oriented surface $\Sigma$ in $\R^3$ with unit normal $N:\Sigma\flecha \S^2$, we will refer to the set $N(\Sigma)\subset \S^2$ as the \emph{Gauss map image} of $\Sigma$.

\begin{definition}
A surface $\Sigma$ in $\R^3$ is a \emph{multigraph} if there is some plane $P$ in $\R^3$ such that $\Sigma$ can be seen locally around each point $p\in \Sigma$ as a graph over $P$. Equivalently, $\Sigma$ is a multigraph if its Gauss map image is contained in an open hemisphere of $\S^2$.
\end{definition}

After a change of Euclidean coordinates, we can always assume without loss of generality for a multigraph $\Sigma$ that $N(\Sigma)$ lies in the upper open hemisphere $\S_+^2$, and so, that $\nu>0$ where $\nu:=\esiz N,e_3\esde$ is the \emph{angle function} of $\Sigma$. Obviously, any graph $z=u(x,y)$ is a multigraph.

Note that the surfaces with vanishing angle function, $\nu \equiv 0$, are open pieces of flat surfaces of the form $\Gamma \times \R$, where $\Gamma$ is some immersed curve in $\R^2$.

\begin{lemma}\label{numulti}
Let $\Sigma$ be an elliptic Weingarten surface, and assume that its angle function satisfies $\nu \geq 0$. Then either $\nu\equiv 0$, or $\nu>0$ on $\Sigma$. If $\nu\equiv 0$, then $\Sigma$ is a piece of a plane or a cylinder.
\end{lemma}
\begin{proof}
Let $f$ be the smooth function defining the relation \eqref{wein2}. If $f$ is not defined at $0$, then by properties (i)-(iv) of $f$ (see Section \ref{sec:21}) it follows that $\Sigma$ has positive curvature. In particular, its Gauss map $N:\Sigma\flecha \S^2$ is a local diffeomorphism, hence an open mapping. Thus, if $\nu\geq 0$, it must actually happen that $\nu >0$, and Lemma \ref{numulti} holds in this case.

Assume next that $f$ is defined at $0$, and let $q_0\in \Sigma$ satisfy $\nu(q_0)=0$. Without loss of generality, we assume that $q_0$ is the origin and $N(q_0)=(1,0,0)$. If $f(0)=0$ (resp. $f(0)\neq 0$), let $C$ denote the vertical plane (resp. the vertical cylinder with principal curvatures $0$ and $f(0)$) that is tangent to $\Sigma$ at $q_0$, with the same orientation. Note that both $\Sigma$ and $C$ satisfy the elliptic Weingarten relation \eqref{wein2}. Thus, both $\Sigma$ and $C$ can be seen around the origin as graphs $x=h_i(y,z)$, $i=1,2$, over their common tangent plane, and $h_1,h_2$ are solutions to the same $C^2$ fully nonlinear elliptic PDE, associated to \eqref{wein2}. Observe that this PDE is of the form \eqref{eedp}-\eqref{haka}. In these conditions, it is well known that the difference $h=h_1-h_2$ satisfies a second order, linear, homogeneous elliptic PDE $L[h]=0$ with $C^1$ coefficients. Note that $Dh(0,0)=(0,0)$, where $Dh=(h_y,h_z)$. Also $h_z=(h_1)_z$, since $h_2(y,z)$ does not actually depend on $z$, because it corresponds to the vertical plane (or cylinder) $C$.

Assume that $h$ is not identically zero. Then, it is a standard fact  (see e.g. Bers \cite{Be}) that there exist coordinates $(u,v)$ obtained by a linear transformation of $(y,z)$ such that $h$ has the local representation around the origin 
 \begin{equation}\label{berfor}
h(u,v)=w(u,v) + o(\sqrt{u^2+v^2})^k
 \end{equation} 
where $w(u,v)$ is a homogeneous harmonic polynomial of degree $k\geq 2$. In particular, the image of $Dh$ cannot lie in a half-plane around the origin. Thus, there exist points arbitrarily close to the origin where $h_z>0$. Hence, $\nu<0$ at those points, since $$\nu=\frac{-(h_1)_z}{\sqrt{1+(h_1)_y^2+(h_1)_z^2}},$$ and $h_z=(h_1)_z$. This contradicts that $\nu\geq 0$ in $\Sigma$. Therefore, $h$ must be identically zero, and so $\Sigma$ is a piece of the cylinder (or plane) $C$. This concludes the proof.
\end{proof}

\subsection{The linearized Weingarten equation}

Let $\Sigma$ be an immersed oriented surface in $\R^3$ with unit normal $N$. Given $\phi\in C_0^{\8}(\Sigma)$, consider the normal variation of $\Sigma$ associated to $\phi$,
 \begin{equation}\label{vari}
(p,\tau)\in \Sigma\times (-\ep,\ep)\mapsto p+\tau \phi(p) N(p),
 \end{equation}
 and denote by $\cH(\tau)$ and $\mathcal{K}(\tau)$ the mean curvature and Gauss curvature of the corresponding surface $\Sigma_{\tau}$ in \eqref{vari}. In \cite{RS} it is shown (see equation (1.1)) that
 \begin{equation}\label{varifor}
 2 \cH'(0)= \Delta \phi + (4H^2-2K)\phi, \hspace{0.5cm} \mathcal{K}'(0)= {\rm div}(T_1 \nabla \phi) + 2HK \phi.
 \end{equation}
Here, $H,K$ denote the mean curvature and Gauss curvature of $\Sigma$; $\Delta, {\rm div}, \nabla$ are the Laplacian, divergence and gradient operator on $\Sigma$, and $T_1:= 2H {\rm Id}-S$, where $S$ is the shape operator of $\Sigma$.

Assume now that $\Sigma$ satisfies an elliptic Weingarten equation \eqref{weq2}. Let $\{\Sigma_{\tau}\}_{\tau\in (-\ep,\ep)}$ be a normal variation of $\Sigma$ associated to some function $\phi\in C_0^{\8}(\Sigma)$, and denote 
 \begin{equation}\label{wt}
\cW(\tau):= \cH(\tau)- g(\cH(\tau)^2 -\mathcal{K}(\tau)):\Sigma \times (-\ep,\ep)\flecha \R,
 \end{equation}
with the previous notation, where $g$ is the function in \eqref{weq2}. Then, taking into account \eqref{varifor}, the linearized operator of the Weingarten equation \eqref{weq2} satisfied by $\Sigma$ is
\begin{equation}\label{operator}
\cW'(0)=\cL_g [\phi] = \left(\frac{1-2g g'}{2}\right) \Delta \phi + g' {\rm div} (T_1 \nabla \phi) + q \phi,
\end{equation}
where $g,g'$ are evaluated at $H^2-K$, and $$q:=2g^2(1-2 g g') -(1-4 g g')K.$$ We remark that $\cL_g$ in \eqref{operator} is a linear elliptic operator, since \eqref{weq2} is elliptic.

\section{Multigraphs with quasiconformal Gauss map}\label{sec:3n}

In this section we prove:

\begin{theorem}\label{alfa}
Planes are the only complete multigraphs with quasiconformal Gauss map and bounded second fundamental form.
\end{theorem}

\begin{remark}\label{2ff}
Let $\Sigma$ be an immersed surface in $\R^3$ with bounded second fundamental form $\sigma$, that is, we have $|\sigma(p)|\leq \sigma_0<\8$ for some constant $\sigma_0$, for all $p\in \Sigma$. This implies a well-known \emph{uniformicity property} for $\Sigma$ as a local graph, see e.g. Proposition 2.3 in \cite{RST}. In our conditions, this property implies that there exists some $\delta=\delta(\sigma_0)>0$ for which the following holds: 

Any $p\in \Sigma$ has a neighborhood $\cW_p\subset \Sigma$ that is a graph over the disk $B_{2\delta}(0)\subset T_p\Sigma$ centered at the origin and of radius $2\delta$ of its tangent plane at $p$. Also, if $u$ denotes the function that defines this graph in any such disk $B_{2\delta}(0)$, it holds $|Du|<1$ in $B_{2\delta}(0)$, where $Du$ denotes the Euclidean gradient of $u$. Moreover, there exists $\mu=\mu(\sigma_0)>0$, such that the $C^2$ norm of $u$ is at most $\mu/2$ on any of the disks $B_{2\delta}(0)$; that is, $||u||_{C^2(B_{2\delta}(0))} <\mu/2$. 

We remark that $\delta,\mu$ only depend on the bound $\sigma_0$ for the second fundamental form of $\Sigma$, and not on $p$ or $\Sigma$. 
\end{remark}

 \begin{proof}
We will argue by contradiction. Let $\Sigma$ be a complete multigraph with bounded second fundamental form, and assume that its Gauss map $N$ is quasiconformal but $\Sigma$ is not a plane. Up to an Euclidean change of coordinates, we can suppose that $N(\Sigma)$ is a subset of $\S_+^2$, and so $\Sigma$ can be locally seen as a graph $z=u(x,y)$.

Take from now on an arbitrary point $p\in \Sigma$, and let $R>0$ be the largest value for which an open neighborhood $\cV\subset \Sigma$ of $p$ can be seen as a graph $z=u(x,y)$ over the disk $D=D(\hat{q},R)$, where $\hat{q}=\pi(p)$, with $\pi(x,y,z):=(x,y)$. See Figure \ref{fig:dicur2}. That this radius $R$ exists, i.e., that $R$ is not infinite, follows from Simon's theorem according to which entire graphs with quasiconformal Gauss map are planes (\cite[Theorem 4.1]{Sim}).

\begin{figure}[htbp]
    \centering
    \includegraphics[width=.8\textwidth]{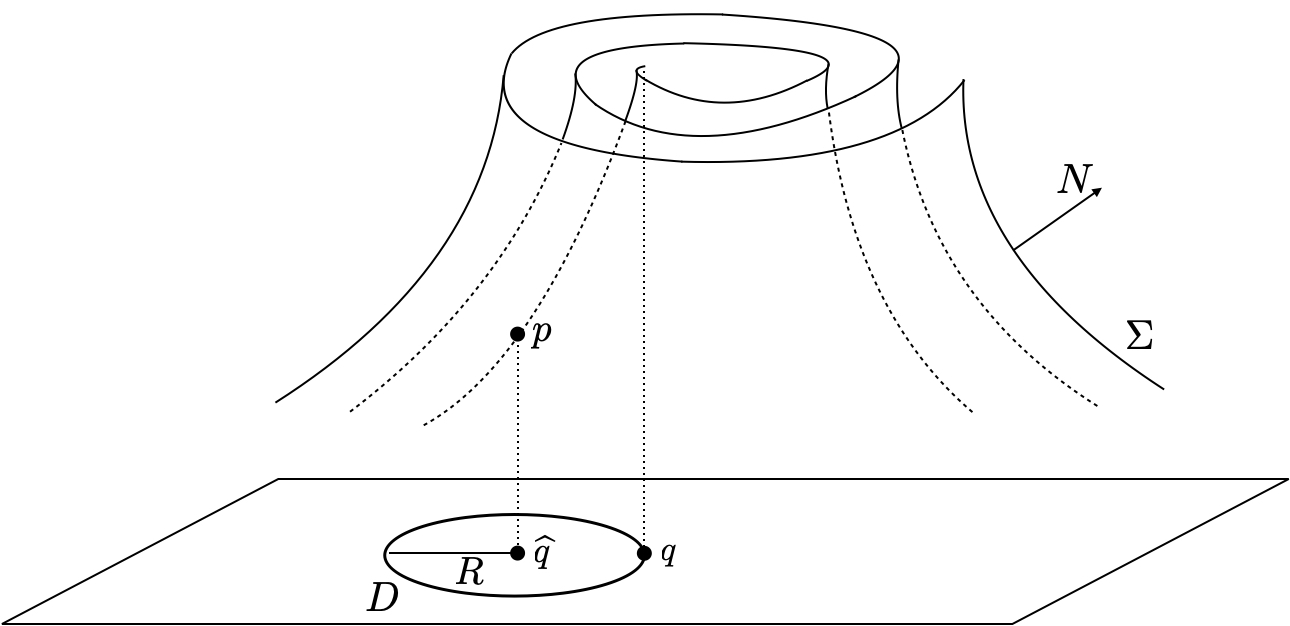}
    \caption{Initial situation of the proof of Theorem \ref{alfa}.}
    \label{fig:dicur2}
\end{figure}

Take $q\in \parc D$ such that $u$ cannot be extended to a neighborhood of $q$, and consider the straight line $\Gamma$ in $\R^2$ that is tangent to $\parc D$ at $q$. 
We will let $\Gamma(s)$ be an arclength parametrization of $\Gamma$, with $\Gamma(0)=q$.

For the rest of the proof, we will let $\delta>0$ denote the constant in Remark \ref{2ff} associated to the bound $|\sigma|\leq \sigma_0<\8$ for the second fundamental form $\sigma$ of $\Sigma$.

Given $s_0\in \R,\ep>0$, we let $\cN(s_0,\ep)$ denote the open \emph{one-sided tubular} set 
\begin{equation}\label{halftu}
\cN(s_0,\ep):= \{\Gamma(s) + \tau\, \eta(s) : |s-s_0|<\delta, \tau \in (0,\ep) \}\subset \R^2,
\end{equation}
 where $\eta(s)$ denotes the unit normal of $\Gamma(s)$ that, at $q$, points in the direction $\hat{q}-q$, where $\hat{q}$ is the center of $D$. See Figure \ref{fig:dentrofuera0}. 

\begin{figure}[htbp]
    \centering
    \includegraphics[width=.5\textwidth]{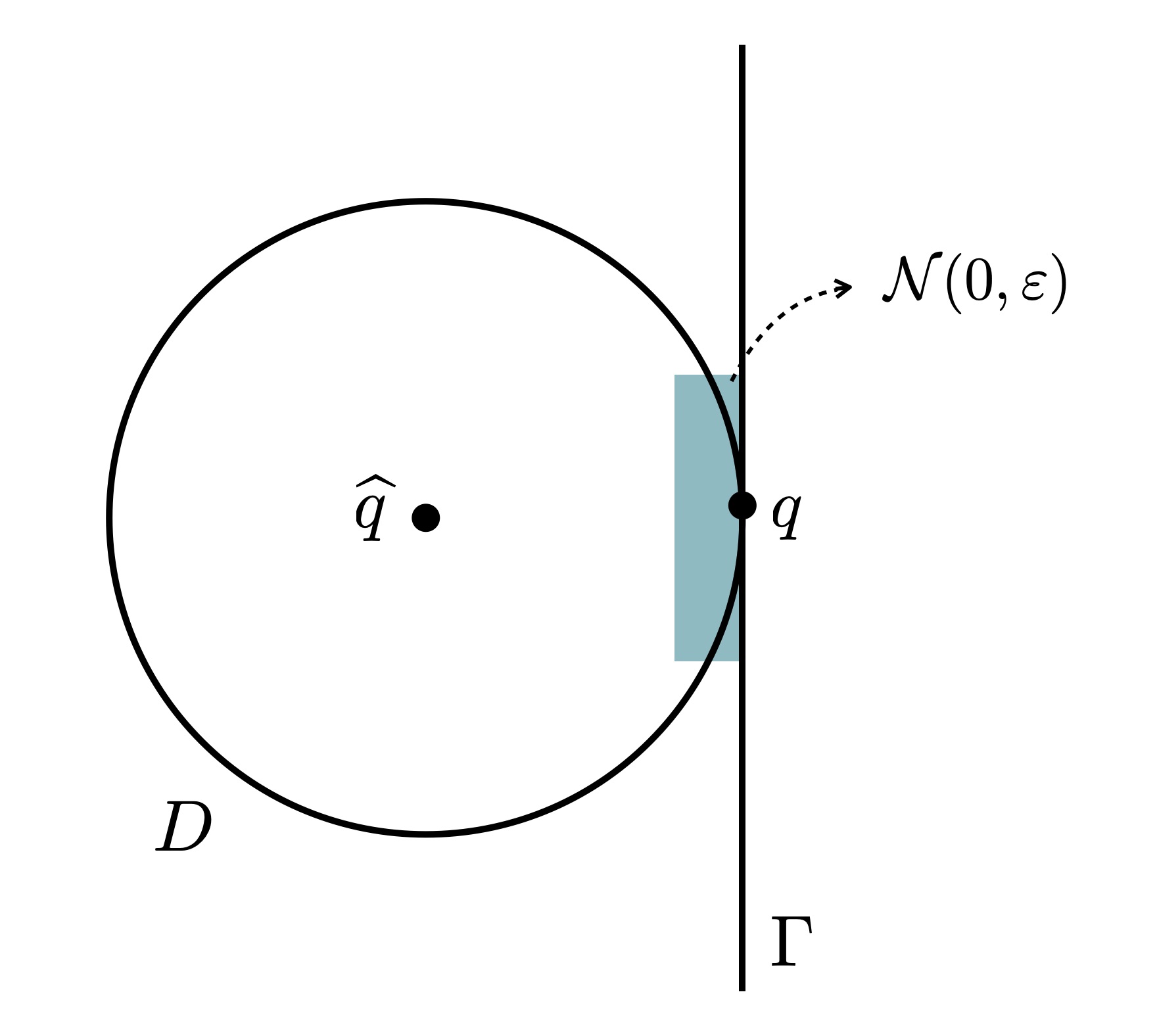}
    \caption{The one-sided tubular domain $\cN(0,\ep)$ at $q$. }
    \label{fig:dentrofuera0}
\end{figure}

\begin{assertion}\label{ass:2}
In the above conditions, there exists $\ep >0$ such that $u(x,y)$ extends smoothly to $D\cup \cN (0,\ep)$.  
Moreover, this extension satisfies that $u(x,y)$ diverges to either $+\8$ or $-\8$ when $(x,y)\in \cN (0,\ep)$ approaches $\Gamma$. 
\end{assertion}
\begin{proof}
The proof is inspired by a similar result of Hauswirth-Rosenberg-Spruck in \cite{HRS} for the case of CMC surfaces in the product space $\H^2\times \R$. Some of the arguments here are however different, since we are under more general conditions. 

Let $\{q_n\}_n\subset D$ converge to $q\in \parc D$, and denote $p_n:=(q_n,u(q_n))\in \cV\subset \Sigma$. 
For each $p'\in \cV$, let $\cW_{p'}\subset \Sigma$ denote the neighborhood of $p'$ that can be seen as a normal graph over $B_{2\delta}(0)\subset T_{p'} \Sigma$, where $\delta>0$ is the one in Remark \ref{2ff}.
Let $\nu:\Sigma\flecha (0,1]$ denote the angle function of $\Sigma$. 

Assume that for some subsequence of the $\{p_n\}$, it happens that $\nu(p_n)\geq \nu_0>0$ for some $\nu_0>0$. Then, the slopes of the planes $T_{p_n}\Sigma$ would be uniformly bounded. By the above property, there would exist neighborhoods $\cW_{p_n}\subset \Sigma$ of the points $p_n$, and some fixed $\ep>0$, with the following property: \emph{for $n$ large enough, each $\cW_{p_n}$ is a vertical graph over the disk in $\R^2$ of center $q_n$ and radius $\ep>0$}. Since $\{q_n\}$ converges to $q$, this contradicts the assumption that $u$ cannot be extended smoothly across $q$.

Therefore, we must have $\nu(p_n)\to 0$, i.e., the tangent planes $T_{p_n} \cV$ become vertical as $n\to \8$. Up to a subsequence, assume that $\{N(p_n)\}_n$ converges to some horizontal vector $N_0\in \parc \S_+^2$.

Denote $\cW_n:= \Phi_n(\cW_{p_n})$, where $\Phi_n$ is the vertical translation of $\R^3$ sending $p_n$ to $(q_n,0)$. Also, let $\Pi_0$ denote the vertical plane of $\R^3$ that contains $q$ and is orthogonal to $N_0$. Since $\{N(p_n)\}_n \to N_0$, then, for $n$ sufficiently large, there exists an open neighborhood $\mathcal{U}_{n}\subset\cW_n$  of $(q_n,0)$ such that:
\begin{enumerate}
\item
$\cU_n$ can be seen as the graph of a function $v_n$ defined on the disk $B_\delta\subset \Pi_0$ centered at $q$ and of radius $\delta$ of the vertical plane $\Pi_0$.
 \item
The $C^2$-norm of the function $v_n$ in $B_\delta$ is at most $\mu$ (by Remark \ref{2ff} and the definition of $\cW_n$).
\end{enumerate}

By standard embedding theorems of Holder spaces (see e.g. \cite[Lemma 6.36]{GT}), it follows that the set $\{v_n\}$ is precompact in the $C^{1,\alfa}$-norm over $B_{\delta}\subset\Pi_0\equiv\R^2$, for any $\alfa\in (0,1)$. Thus, a subsequence of the $v_n$ converges uniformly in the $C^{1,\alfa}(B_{\delta})$-norm to some function $v^0 \in C^{1,\alfa} (B_{\delta})$. So, the surfaces $\cU_n$ converge (up to subsequence) in the $C^{1,\alfa}$-norm in compact sets to some limit surface in $\R^3$, which is a $C^{1,\alfa}$-graph over $B_{\delta}$ . Obviously, the Gauss maps $N_n$ of the surfaces $\cU_n$ are quasiconformal (note that each $\cU_n$ is a translation of an open subset of $\Sigma$, that has quasiconformal Gauss map). 

Let us now reparametrize the topological disks $\cU_n$ and the limit surface in graphical coordinates over $B_{\delta}$. Up to a homothety in these coordinates, we can then view each $\cU_n$ as a smooth map $\psi_n:\D\flecha \R^3$ from the unit disk $\D$, so that they converge in the $C^{1,\alfa}$-norm to the limit immersion $\tilde{\psi}:\D\flecha \R^3$, which is of class $C^{1,\alfa}$. Call $E_n,F_n,G_n$ and $\tilde{E}, \tilde{F},\tilde{G}$ to the coefficients of the metrics $\esiz d\psi_n,d\psi_n\esde $ and $\esiz d\tilde{\psi},d\tilde{\psi}\esde $ with respect to the coordinates $(x_1,x_2)$ of $\D$. Note that $(E_n,F_n,G_n)\to (\tilde{E}, \tilde{F},\tilde{G})$ in the $C^{0,\alfa}$-norm. Then, we can introduce conformal parameters on $\D$ for these metrics, by means of classical uniformation theorems for the Beltrami equation. Specifically, if we let $z=x_1+ix_2$, we can write $\esiz d\psi_n,d\psi_n\esde = \rho_n |dz + \mu_n d\bar{z}|^2$, where 
 \begin{equation}\label{cambiocon}
\rho_n = \frac{1}{4}(E_n+G_n+2\sqrt{E_n G_n - F_n^2}), \hspace{0.5cm} \mu_n =\frac{E_n -G_n +2i F_n}{4 \rho_n}.
 \end{equation} 
In this way, $||\mu_n||_{\8}<1$, and the change of coordinates $(x_1,x_2)\mapsto (u^n,v^n)\in \D$ into conformal (isothermal) coordinates $(u^n,v^n)$ for $\esiz d\psi_n,d\psi_n\esde$ is given by a homeomorphic solution $f^n = u^n+i v^n :\D\flecha \D$ to the Beltrami equation $$f^n_{\bar{z}} = \mu_n f^n_z.$$ By classical theory of quasiconformal mappings, this solution $f^n$ exists and is unique if we prescribe that $0,1$ are fixed points, i.e. that $f^n(0)=0$ and $f^n(1)=1$; see e.g. \cite[Theorem 6]{AB}. A similar discussion provides conformal parameters for the limit immersion $\tilde{\psi}$, and it is clear by the previous convergence properties and \eqref{cambiocon} that the Beltrami coefficients $\{\mu_n\}$ converge uniformly on compact sets to the Beltrami coefficient $\tilde{\mu}$ associated to $\esiz d\tilde{\psi},d\tilde{\psi}\esde $ in the previous process. By the regular dependence of $f^n$ with respect to $\mu_n$ proved in \cite{AB}, and the convergence properties in our situation, we conclude that the mappings $\{f^n\}_n$ converge to the homeomorphic solution $\tilde{f}$ of the Beltrami equation that provides conformal parameters for $\esiz d\tilde{\psi},d\tilde{\psi}\esde $ and fixes $0$ and $1$.

To sum up: this discussion shows that we can view the convergence of the surfaces $\cU_n$ as a convergence of \emph{conformal} immersions $\psi_n:\D\flecha \R^3$ to some limit conformal immersion $\tilde{\psi}:\D\flecha \R^3$ in the $C^{1,\alfa}$-norm. In particular, the Gauss maps $N_n:\D\flecha \S^2$ of the $\psi_n$ converge to the Gauss map $\tilde{N}:\D\flecha \S^2$ of $\tilde{\psi}$.

By the discussion in Section \ref{sec:cuasi}, the stereographically projected maps $g_n:=\pi \circ N_n$ satisfy that either $g_n$ or $g_n^*(z):=g_n(\bar{z})$ is a quasiregular mapping. By Lemma \ref{lemcuasii}, these maps converge, up to a subsequence, to a quasiregular mapping $\tilde{g}$, maybe constant. This mapping is actually given by $\tilde{g}=g$ or $\tilde{g}=g^*$, where $g=\pi\circ \tilde{N}$ and
 $g^*(z):=g(\bar{z})$, by the previous discussion. Since $\nu(p_n)\to 0$ and $\nu> 0$, we have that $\esiz \tilde{N},e_3\esde \geq 0$ everywhere, and $\esiz \tilde{N},e_3\esde =0$ at $q$; thus, $\tilde{N}$ cannot be open. Since non-constant quasiregular mappings are open by Stoilow's factorization theorem, $\tilde{N}$ must actually be constant. Thus, this limit surface $\tilde{\psi}$ is a piece of a plane; more specifically, it is the disk $B_{\delta}$ of the vertical plane $\Pi_0\subset \R^3$ that passes through $q$ with unit normal $N_0$.

We prove next that $N_0$ is orthogonal to $\parc D$ at $q$. Indeed, otherwise, there would exist points of the disk $B_{\delta}\subset \Pi_0$ that lie in the horizontal disk $D$. Obviously, the function $u$ would be well defined around any such point, since it is well defined in $D$. Take a point $(a_0,0)$ in $B_{\delta}\cap D$. By the already proved convergence of the disks $\cU_n\subset \Sigma$ to $B_{\delta}$, there would exist points in $D$ converging to $(a_0,0)$ at which the gradient of $u$ blows up, since $B_{\delta}$ is vertical. This is not possible, since $u$ is well defined at $(a_0,0)$. Therefore, $N_0=\pm \eta(0)$, where $\eta(s)$ was defined in \eqref{halftu}, and so, $\Pi_0=\Gamma\times \R$.

Take now a small segment $\beta$ contained in $D$, that ends at $q$. Since the tangent planes of $\cV$ become vertical as we approach $q$ through $\beta$, it is clear that the restriction of $u$ to $\beta$ is monotonic, for values sufficiently close to $q$. Thus, it has a limit, which cannot be a finite number. Indeed, if it was a finite number, then by monotonicity, the curve $(\beta,u(\beta))$ would have finite length in $\Sigma$. So, by completeness of $\Sigma$, we would have a point in $\Sigma$ with a vertical tangent plane, and this contradicts that $\Sigma$ is a multigraph. In other words, the restriction of $u$ to any such segment $\beta$ diverges to $+\8$ or to $-\8$.

Consider the normal segment in $\R^2$ orthogonal to $\parc D$ at $q$, given by $\beta(t):=q\pm t N_0$, where the sign is chosen so that $\beta(t)$ lies in $D$ for $0<t< t_0$ with $t_0$ small enough. Define the open set 
\begin{equation*}
\Sigma_{t_0}=\bigcup_{0<t<t_0}{\cal W}_{(\beta(t),u(\beta(t)))} \subset \Sigma,
\end{equation*}
which is a connected neighborhood of the curve $\{(\beta(t),u(\beta(t))):\ 0<t< t_0\}\subset\Sigma$; here, one should recall the definition of $\cW_p$ at the beginning of the present proof of Assertion \ref{ass:2}.  By the convergence of the disks $\cU_n\subset \Sigma$ to $B_{\delta}\subset \Gamma\times \R$, it is clear that 
the projection of $\Sigma_{t_0}$ into $\R^2$ contains a one-sided tubular domain $\cN(0,\ep_0)$ as in \eqref{halftu}. See Figure \ref{fig:appendix}.

 \begin{figure}[htbp]
    \centering
        \includegraphics[width=.5\textwidth]{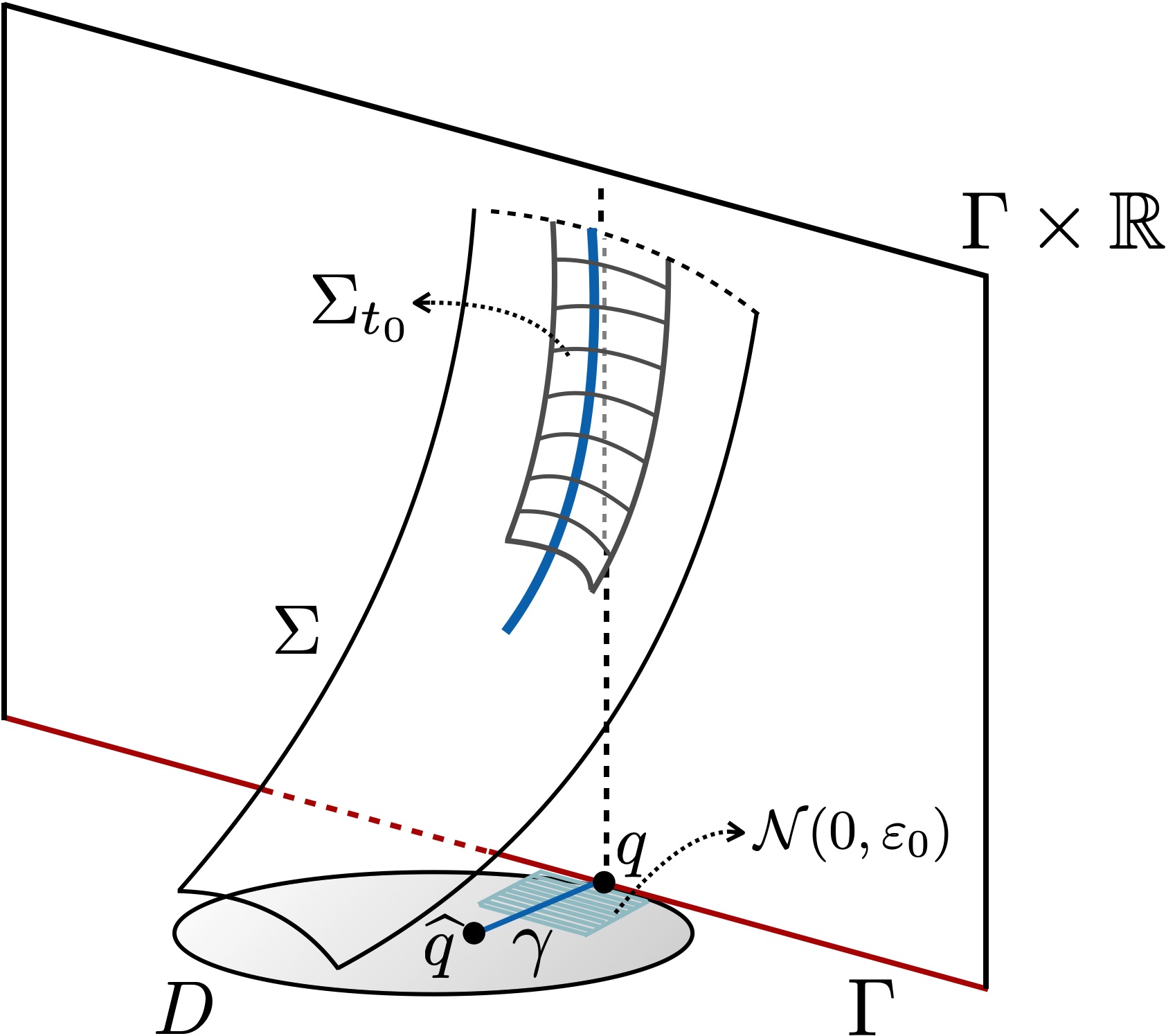}
    \caption{The domain $\Sigma_{t_0}$ in $\Sigma$, and its projection.}
    \label{fig:appendix}
\end{figure}

For each $s\in [-\delta,\delta ]$, denote by $P(s)$ the vertical plane that is normal to $\Gamma$ at the point $\Gamma(s)$. Observe that $P(s)$ intersects $\Sigma_{t_0}$ transversely for all $s\in[-\delta,\delta]$. Note that all points in $\Sigma_{t_0}\cap P(0)$ belong to the curve $(\beta(t),u(\beta(t)))$, and so $\Sigma_{t_0}\cap P(0)$ is a connected graphical curve that does not intersect $\Gamma\times \R$. In the same way, by transversality and the definition of $\Sigma_{t_0}$, there is some $t_0>0$ and some $\ep>0$ such that for each $s\in [-\delta,\delta]$, $\Sigma_{t_0}\cap P(s)$ is exactly one curve, which is a graph over a segment in $\R^2$ of the form $\Gamma(s) + t \eta(s)$. Here, $\eta(s)$ is the unit normal of $\Gamma(s)$ in \eqref{halftu} and $t$ varies in an interval $I_s$ that contains $(0,\ep)$.

All these properties let us conclude that $\Sigma_{t_0}$ is a graph when we restrict to the points of $\Sigma_{t_0}$ that project onto the one-sided tubular domain $\cN(0,\ep)$, for the value $\ep>0$ above. Thus, $u$ can be extended as a graph to $D\cup \cN(0,\ep)$.

Let us also point out that (for $t_0>0$ small enough) $\Sigma_{t_0}$ does not intersect $\Gamma\times \R$. Indeed, otherwise there would exist a smallest (in absolute value) $s_1$ such that $\Sigma_{t_0}\cap P(s_1)$ intersects $\Gamma\times \R$. But $\Sigma_{t_0}\cap P(0)$ does not intersect $\Gamma\times \R$, as explained above, so $s_1>0$. By continuity we would have that $\Sigma_{t_0}\cap P(s_1)$ intersects $\Gamma\times \R$ but it does not \emph{cross} it. Hence,  there would exist a point in $\Sigma_{t_0}\cap P(s_1)$ where the tangent plane to $\Sigma$ is vertical, and this is not possible since $\Sigma$ is a multigraph.

The fact that $\Sigma_{t_0}$ does not intersect $\Gamma\times \R$ together with the previously proved asymptotic convergence of the curves $\Sigma_{t_0}\cap P(s)$ to $\Gamma \times \R$  give the asymptotic behavior of the statement of Assertion \ref{ass:2}. This completes the proof of Assertion \ref{ass:2}.
\end{proof}

Using the notation of Assertion \ref{ass:2}, take $q'\in \cN(0,\ep)\cap D$ given by $q'=\Gamma(\delta/2)+\tau_0\eta (\delta/2)$ for $\tau_0\in (0,\ep)$ small enough, and let $p'=(q',u(q'))\in \Sigma$. Then, we can apply again the extension process in Assertion \ref{ass:2}, but this time starting with $p'$ instead of $p=(\hat{q},u(\hat{q}))$. In this way, we see that $u$ can also be extended to a one-sided tubular domain $\cN(\delta/2,\ep_1)$, i.e., to
 \begin{equation}\label{larges}
\{\Gamma(s) + \tau\, \eta(s) : s\in [-\delta/2,3\delta/2], \tau\in (0,\ep_1)\}\subset \R^2
 \end{equation}for some $\ep_1>0$, and in particular to the union $\cN(0,\ep)\cup \cN(\delta/2,\ep_1)$. By repeating this process, we conclude that $u$ can be extended to a union of domains $\cN(k\delta/2,\ep_k)$ for $k\in \Z$, and so, to a one-sided tubular neighborhood $\cN_{\Gamma}$ of the straight line $\Gamma\subset \R^2$. Moreover $u(x,y)\to \pm \8$ as $(x,y)\to \Gamma$. See Figure \ref{fig:tri}, left.
 
 We claim next that $u$ can actually be extended to the slab $\cS_{\Gamma}$ of $\R^2$ contained between $\Gamma$ and the line parallel to $\Gamma$ that passes through $\hat{q}$. To see this, take for each $\theta\in (-\pi/2,\pi/2)$ the open segment $\sigma_{\theta}$ that joints $\hat{q}$ with $\Gamma$ and that makes an angle $\theta$ with the segment $\sigma_0$ that joints $\hat{q}$ and $q$. Note that $u$ is well defined on $\sigma_{\theta}$ for sufficiently small values of $\theta$. Let $\theta_0>0$ be the supremum of the values for which $u$ can be extended to the open triangular region (see Figure \ref{fig:tri}, left) $$\Omega_0:= \{\cup \sigma_{\theta} : |\theta|<\theta_0\}.$$
 \begin{figure}[htbp]
    \centering
    \includegraphics[width=.8\textwidth]{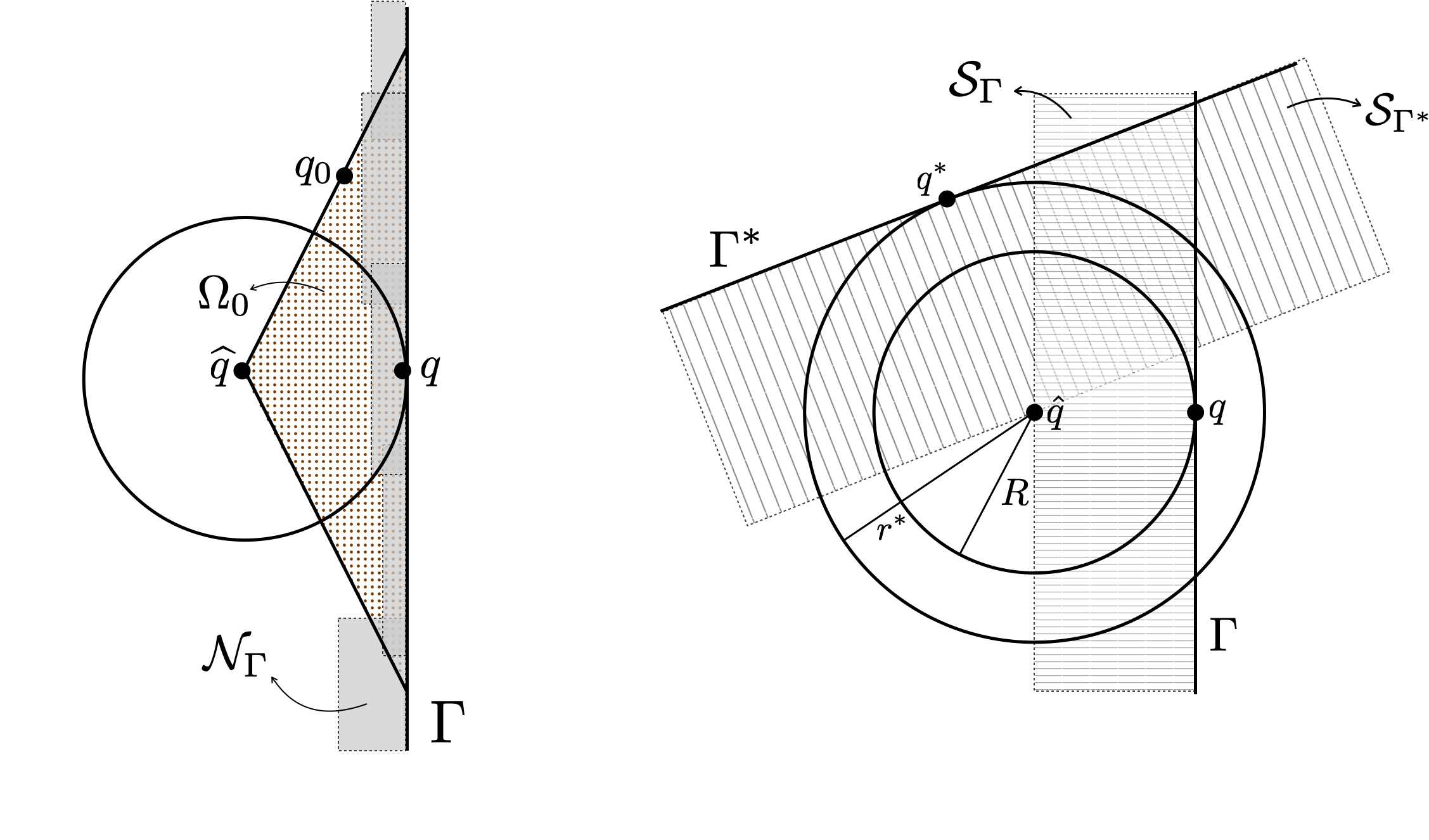} 
    \caption{Left: the one-sided tubular neighborhood $\cN_{\Gamma}$ of the line $\Gamma$, and the triangular region $\Omega_0$. Right: the union of the two strips $S_{\Gamma}$ and $S_{\Gamma^*}$.}
    \label{fig:tri}
\end{figure}
If $\theta_0\neq \pi/2$, there exists $q_0\in \sigma_{\theta_0}\subset \parc \Omega_0$ such that $u$ cannot be extended around $q_0$. By the previous process, we see that $u\to \pm \8$ along the whole segment $\sigma_{\theta_0}$. But this is a contradiction, since the open segment $\sigma_{\theta_0}$ intersects the one-sided tubular set $\cN_{\Gamma}$, where $u$ is well defined.  Therefore, $\theta_0=\pi/2$, and this means that $u$ can be extended to the slab $\cS_{\Gamma}$.

Next, let $P^+$ denote the open half-plane of $\R^2$ with $\parc P^+=\Gamma$ and $\hat{q}\in P^+$. Assume that $u$, which is at first defined on an open subset of $P^+$, cannot be globally extended to $P^+$. Then, there exists some $r^*\geq R$ and some $q^* \in P^+\cap \parc D(\hat{q},r^*)$ such that $u$ is well defined on $P^+\cap D(\hat{q},r^*)$ but cannot be smoothly extended around $q^*$. Then, by repeating the previous argument, but this time with respect to $q^*$ (instead of $q$), we deduce that $u$ is well defined in the slab $\cS_{\Gamma^*}$ between the line $\Gamma^*$ that is tangent to $\parc D(\hat{q},r^*)$ at $q^*$, and the line parallel to $\Gamma^*$ that passes through $\hat{q}$. Moreover $u\to \pm \8$ as we approach $\Gamma^*$. See Figure \ref{fig:tri}, right.

Observe here that $\Gamma^*$ needs to be parallel to $\Gamma$. Indeed, otherwise the union of the slabs $\cS_{\Gamma}\cup \cS_{\Gamma^*}$ is a simply connected domain in $\R^2$ where $u$ is globally well defined, but this is impossible since $u\to \pm \8$ as we approach $\Gamma^*$, which actually intersects $\cS_{\Gamma}$.

Then, it clearly follows from this argument that $u$ can be extended to a domain $\Omega\subset \R^2$ that is either a half-plane or a strip between two parallel lines, and so that $u(x,y)\to \pm \8$ as $(x,y)\to \parc \Omega$. In particular, the complete multigraph $\Sigma$ is actually the graph $z=u(x,y)$ over $\Omega$.

Finally, let us recall that the Gauss map $N:\Sigma\flecha \S^2$ of $\Sigma$ is quasiconformal. By Theorem 3.1 in \cite{Sim} (see also equation (3.26) in \cite{Sim}), and since $\Sigma$ is a graph, there exist constants $c'>0$ and $\alfa\in (0,1)$ such that 
\begin{equation}\label{simeq}
||N(x)-N(\bar{x})|| \leq c' \left(\frac{||x-\bar{x}||}{\varrho}\right)^{\alfa},
\end{equation}
for all $x,\bar{x}\in \Sigma$ that are at an extrinsic distance at most $\varrho/2$ from some arbitrary point $x_0\in \Sigma$; here, $\varrho>0$ and $||\cdot ||$ denotes the Euclidean distance in $\R^3$. Since $u\to \pm \8$ as $(x,y)\to \parc \Omega$, we deduce that $\Sigma$ is proper; hence, by letting $\varrho\to \8$ in \eqref{simeq} we conclude that $N$ must be constant, i.e., $\Sigma$ must be a plane, a contradiction (recall that $\Omega$ is not $\R^2$). This completes the proof of Theorem \ref{alfa}.
 \end{proof}

\section{Weingarten multigraphs with bounded second fundamental form}\label{sec:3}

The present section is devoted to the proof of the following theorem:
\begin{theorem}\label{b2}
Planes are the only complete elliptic Weingarten multigraphs with bounded second fundamental form.
\end{theorem}
\begin{proof}
Let $\Sigma$ be a complete elliptic Weingarten multigraph with bounded second fundamental form. Let $f\in C^{2}(I_f)$, $I_f\subset \R$, be the function that defines the Weingarten relation \eqref{wein2} satisfied by $\Sigma$. 

We will start by noting that $f$ is defined at $0$, i.e., that $0\in I_f$. Indeed, otherwise we would have $I_f\subset (0,\8)$, by our convention that the umbilicity constant $\alfa$ of $f$ satisfies $\alfa \geq 0$ (see Section \ref{sec:21}). Then, by the properties of $f$ described in Section \ref{sec:21} we deduce that $I_f=(a,\8)$ for some $a\geq 0$. By the symmetry condition $f\circ f={\rm Id}$, we see that
both principal curvatures of $\Sigma$ are positive. Since $\Sigma$ has bounded second fundamental form, we easily see from there that its Gaussian curvature satisfies $K\geq c>0$ for some constant $c$. Thus, $\Sigma$ is compact, a contradiction with $N(\Sigma)\subset \S_+^2$.

So, there exists $f(0)$. If $f(0)=0$, $\Sigma$ has quasiconformal Gauss map (see Lemma \ref{lem:cuw2}). By Theorem \ref{alfa}, $\Sigma$ is a plane. If $f(0)\neq 0$, $\Sigma$ cannot exist, by Theorem \ref{ass:4} below. Thus, Theorem \ref{b2} is proved.
\end{proof}

So, it remains to prove:

\begin{theorem}\label{ass:4}
There are no complete elliptic Weingarten multigraphs with $f(0)\neq 0$ and bounded second fundamental form.
\end{theorem}
\begin{proof}
Let $\cW_f$ denote the class of all immersed oriented surfaces in $\R^3$ that satisfy \eqref{wein2} for our choice of $f$, with $f(0)\neq 0$. By ellipticity, surfaces in $\cW_f$ satisfy the maximum principle. 

We start similarly to the proof of Theorem \ref{alfa}. Arguing by contradiction, let $\Sigma$ be a complete elliptic Weingarten multigraph with bounded second fundamental form and $f(0)\neq 0$. Note that we actually have $f(0)>0$ (by monotonicity, since $\alfa> 0$ in this case; recall that $\alfa>0$ is the umbilical constant associated to $f$, defined by $f(\alfa)=\alfa$). Up to a Euclidean change of coordinates, we assume that $N(\Sigma)$ is a subset of $\S_+^2$, and so the angle function $\nu$ of $\Sigma$ is positive.

Let $p$ be an arbitrary point $p\in \Sigma$, and let $R>0$ be the largest value for which an open neighborhood $\cV\subset \Sigma$ of $p$ can be seen as a graph $z=u(x,y)$ over $D=D(\hat{q},R)$, where $\hat{q}=\pi(p)$, with $\pi(x,y,z):=(x,y)$. See Figure \ref{fig:dicur}. 

\begin{assertion}\label{ass:1}
In the previous conditions, $R\leq 1/\alfa <\8$.
\end{assertion}
\begin{proof}
Since $\alfa>0$, any sphere $S_{\alfa}$ of radius $1/\alfa$ lies in the Weingarten class $\cW_f$ for its inner orientation. In case $R>1/\alfa$, we could place a sphere $S_{\alfa}$ above the graph $\cV$, and then move it downwards until reaching a first contact at an interior point of both surfaces. This is a contradiction with the maximum principle for surfaces in $\cW_f$. This proves Assertion \ref{ass:1}.
\end{proof}

Let us fix next some notation for the rest of the proof of Theorem \ref{b2}. One should compare it with the related notation in the proof of Theorem \ref{alfa}.

By Assertion \ref{ass:1} there exists some $q\in \parc D$ for which $u$ cannot be extended to a neighborhood of $q$. We let $C_1:=\Gamma_1\times \R$, $C_2:=\Gamma_2\times \R$, denote the two vertical cylinders in $\R^3$ of radius $1/f(0)$ that pass through $q$, and whose unit normals at $q$ are orthogonal to $\parc D$. Note that $C_i \in \cW_f$, $i=1,2$, for their inner orientation. We will let $\Gamma_i(s)$ be an arclength parametrization of the circle $\Gamma_i$, with $\Gamma_i(0)=q$.

Given $s_0\in \R,\ep>0$, we define for each $i=1,2$, in analogy with \eqref{halftu},  the open \emph{one-sided tubular} set 
\begin{equation}\label{halftub}
\cN_i(s_0,\ep):= \{\Gamma_i(s) + \tau\, \eta_i(s) : |s-s_0|<\delta, \tau \in (0,\ep) \}\subset \R^2,
\end{equation}
 where, again, $\eta_i(s)$ denotes the unit normal of $\Gamma_i(s)$ that, at $q$, points in the direction $\hat{q}-q$, where $\hat{q}$ is the center of $D$. See Figure \ref{fig:dentrofuera}.
 
\begin{figure}[htbp]
    \centering
    \includegraphics[width=.7\textwidth]{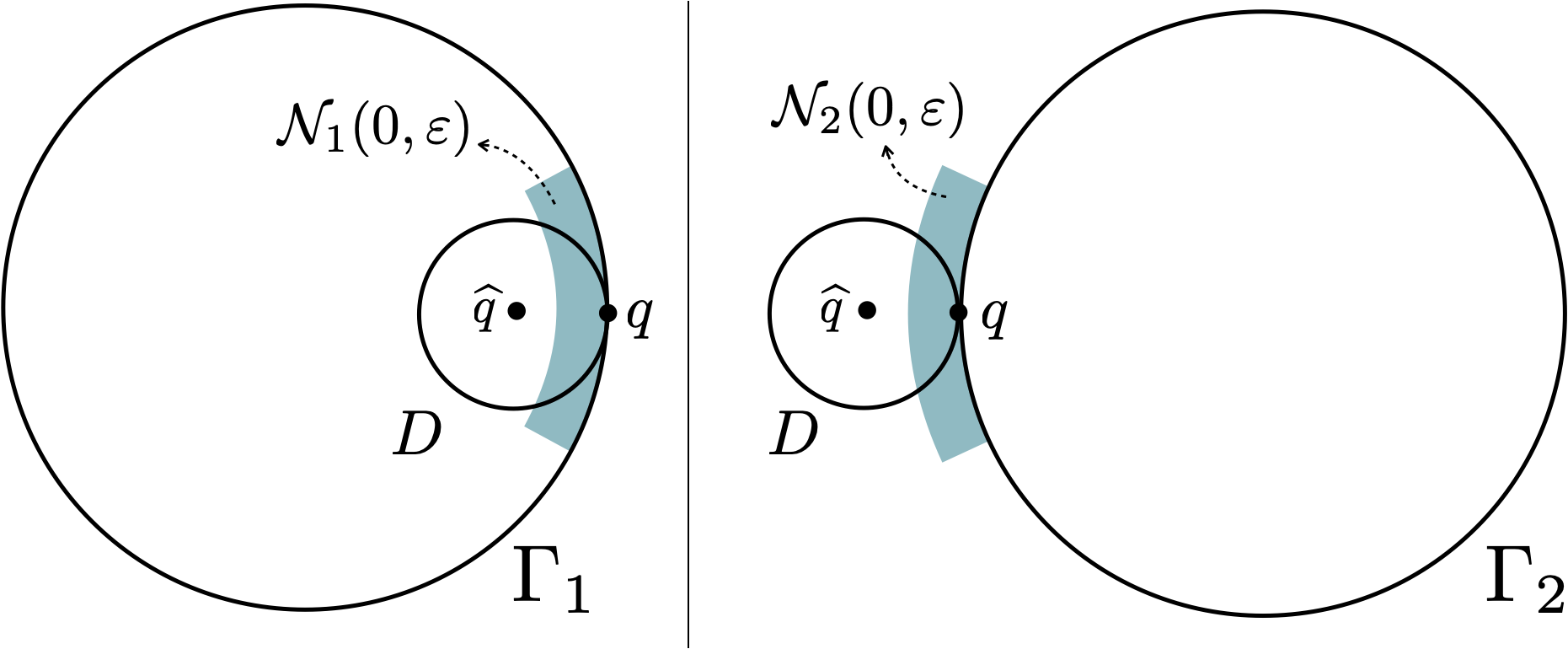}
    \caption{The two circles $\Gamma_i$, $i=1,2$, that are tangent to $\parc D$ at $q$, and their associated one-sided tubular domains $\cN_i(0,\ep)$ at $q$.}
    \label{fig:dentrofuera}
\end{figure}

\begin{assertion}\label{ass:2H}
In the above conditions, there exists $\ep >0$ such that $u(x,y)$ extends smoothly to $D\cup \cN_i (0,\ep)$, for some $i=1,2$. Moreover, this extension satisfies that $u(x,y)$ diverges to either $+\8$ or $-\8$ when $(x,y)\in \cN_i (0,\ep)$ approaches $\Gamma_i$.
\end{assertion}
\begin{proof}
The proof of Assertion \ref{ass:2H} follows very closely the one of Assertion \ref{ass:2}, but with a different convergence argument. Let us explain in detail how to modify the proof of Assertion \ref{ass:2} to our context.

To start, the first five paragraphs in the proof of Assertion \ref{ass:2}, including properties (1) and (2) there, are literally the same. After properties (1) and (2), we should include for our present situation the following paragraph regarding the convergence of the topological disks $\cU_n$ to a limit disk with vanishing angle function, instead of the argument in Assertion \ref{ass:2} using quasiregular mappings:

\emph{Since each $\cU_n$ satisfies the Weingarten equation \eqref{wein2}, then the graphing function $v_n$ of $\cU_n$ is a uniformly elliptic solution to the associated PDE \eqref{eedp}. Recall that the functions $v_n$ have uniformly bounded $C^2$-norm in $B_{\delta}$. Then, by Nirenberg's a priori $C^{2,\alfa}$-estimates for fully nonlinear elliptic equations in dimension two (see \cite[Theorem I]{Nir}), it follows that the family $\{v_n\}_n$ is uniformly bounded in the $C^{2,\alfa}$-norm in $B_{\delta'}$, for any $\delta'\in (0,\delta)$. From here, by the Arzela-Ascoli theorem,  
the surfaces $\cU_n$ converge (up to subsequence) in the $C^2$-norm on compact sets to some limit surface that also satisfies \eqref{wein2}. Since $\nu(p_n)\to 0$ and $\nu\geq 0$, it follows from Lemma \ref{numulti} that this limit surface is a piece of the cylinder $\Gamma\times \R$, where $\Gamma$ is the circle of radius $1/f(0)$ that passes through $q$ with interior unit normal $N_0$. More specifically, this limit surface is the geodesic disk of $\Gamma\times\R$ centered at $q$ and of radius $\delta>0$}.

Once we know this convergence, the same argument as the corresponding one in Assertion \ref{ass:2} shows that $N_0$ is orthogonal to $\parc D$ at $q$. Therefore, $\Gamma$ must be tangent to $\parc D$ at $q$, i.e., $\Gamma=\Gamma_i$ for some $i\in \{1,2\}$. From this point on, the rest of the proof of Assertion \ref{ass:2H} follows literally the corresponding proof of Assertion \ref{ass:2}. In this respect, one should bear in mind that, this time, the curve $\Gamma$ in \eqref{halftub} is a circle, not a line, and that $B_{\delta}\subset \Gamma\times \R$ should be understood as the limit geodesic disk of the topological disks $\cU_n$, and not as a disk of the vertical plane $\Pi_0$.
 \end{proof}

\begin{assertion}\label{ass:3H}
In the conditions of Assertion \ref{ass:2H}, we have that $u\to \8$ (resp. $u\to -\8$) if $D$ lies in the interior (resp. exterior) of $\Gamma_i$.
\end{assertion}
\begin{proof}
This is a direct consequence of the fact that $\Sigma$ was oriented so that its angle function is positive, and the property that $u(x,y)$ diverges to $\pm \8$ proved in Assertion \ref{ass:2H}.
\end{proof}

\begin{remark}\label{ciligraf}
We point out the following consequence of Assertion \ref{ass:2H}, for later use. Assume that we are in the conditions of Assertion \ref{ass:2H}, and in particular $f(0)> 0$. Call $\Gamma:=\Gamma_i$, and suppose, for definiteness, that $u\to -\8$. Then, by the asymptotic behavior of $u$, and taking a smaller $\ep>0$ if necessary, the graph $\cU_0\subset \Sigma$ given by $z=u(x,y)$ on $\cN_i (0,\ep)$ can be seen as a normal graph over an open set $\mathcal{C}_0$ of the limit cylinder $\Gamma\times \R$, of the form $$\mathcal{C}_0=\{(\Gamma(s),t) : |s|<\delta, t\in (-\8,t_0(s))\}\subset \Gamma\times \R,$$ where $t_0(s)$ is a continuous function on $[-\delta,\delta]$. Moreover,  $\cU_0$ lies in the exterior region of $\Gamma\times \R$ (since $u\to -\8$, see Assertion \ref{ass:3H}), and converges asymptotically to $\Gamma\times \R$ as $t\to -\8$.
\end{remark}

We continue with the proof of Theorem \ref{ass:4}, with the above notations. We will also let $\Gamma$ be the circle $\Gamma_i$ in Assertion \ref{ass:2H} (i.e., either $\Gamma_1$ or $\Gamma_2$); note that $\Gamma$ has radius $r_0=1/f(0)$.

To start, consider the graph $\cU_0\subset \Sigma$ given by $z=u(x,y)$ in the \emph{small} one-sided tubular set $\cN(0,\ep)$, defined as in \eqref{halftub}. Note that this time we cannot apply Assertion \ref{ass:2H} recursively as we did with Assertion \ref{ass:2} in \eqref{larges} to extend $u(x,y)$ to a one-sided tubular neighborhood of $\Gamma$, since now $\Gamma$ is a circle (not a straight line), and hence not simply connected. To avoid this difficulty, we will adapt to our situation a perturbation argument by Espinar and Rosenberg \cite{ER}, originally developed for the case of CMC surfaces in Riemannian product spaces $M^2\times \R$. 

First, we will suppose from now on, for definiteness, that \emph{$\cN(0,\ep)$ lies in the exterior of $\Gamma$}, as in the right picture of Figure \ref{fig:dentrofuera} (the argument is basically the same if $\cN(0,\ep)$ lies inside the circle $\Gamma$). That is, we assume that $\cU_0$ lies in the exterior of $\Gamma\times \R$.

Let $S_0$ be the universal cover of the cylinder $\Gamma\times \R$, parametrized  by
 \begin{equation}\label{unicor}
(s,t)\in \R^2 \mapsto (\Gamma(s),t) \in \Gamma\times\R.
\end{equation}

Consider the universal cover of $\R^3$ minus the axis of $\Gamma \times \R$, and choose there the natural \emph{cylindrical coordinates} $(s,t,\rho)$, so that $\rho$ gives the distance to the axis of $\Gamma\times \R$, and $(s,t)$ correspond to the parameters in \eqref{unicor}. In particular, $S_0$ corresponds to the horizontal plane $\rho=r_0$.

Then, by Remark \ref{ciligraf}, the surface $\cU_0\subset \Sigma$ lifts to a graph $\rho=v(s,t)$ over an open set of the plane $S_0$, of the form $\{(s,t) : |s|<\delta, t<t_0(s)\}$ for $t_0(s):[-\delta,\delta]\flecha \R$ continuous. Moreover, this graph lies above $S_0$ (since $\cU_0$ lies in the exterior of $\Gamma\times \R$), and converges asymptotically to $S_0$ as $t\to -\8$.

Once here, we can make an extension process with respect to the variable $s$, similar to the one that we performed in \eqref{larges}, but this time with respect to the coordinates $(s,t,\rho)$. In this way, we obtain that a certain subset of $\Sigma$ lifts to a graph $\rho=w(s,t)$ over a domain $\tilde{\Omega}\subset S_0$ of the form $\{(s,t): s\in\R, t< t_0(s)\}$, for some continuous function $t_0(s)$ on $\R$. Call $M$ to this graph. Note that $M$ lies above $S_0$ and converges asymptotically to $S_0$ as $t\to -\8$. See Figure \ref{fig:cili}.

We now make a deformation argument on the universal cover $S_0$ of $\Gamma\times\R$.

Let us parametrize $S_0$ as in \eqref{unicor}. Then, its first and second fundamental forms are $I=ds^2+dt^2$ and $II = 2H_0 ds^2.$ Note that $2H_0=\kappa_1=1/r_0$, a constant positive value, and $K=0$. If we write the Weingarten equation satisfied by $\Sigma$ (and by $S_0)$ as in \eqref{weq2}, then a computation shows that the linearized operator $\cL_g$ given by \eqref{operator} is written on $S_0$ with respect to the flat parameters $(s,t)$ by
\begin{equation}\label{opeci}
\cL_g [\phi] = A \phi_{ss} + B \phi_{tt} + C \phi,
\end{equation}
for any $\phi\in C_0^{\8} (S_0)$, where $A,B,C$ are the constants $$A=\frac{1}{2}(1-2g(H_0^2)g'(H_0^2)), \hspace{0.5cm} B=A+2H_0g'(H_0^2),  \hspace{0.5cm} C= 4AH_0^2.$$ We remark that $A,B>0$ (by ellipticity of \eqref{weq2} and thus of $\cL_g$), and so $C>0$ as well.

Let $\Omega_0\subset S_0$ be the compact domain parametrized by $(s,t)\in [-L,L]\times [-r,r]$ for some fixed arbitrary values $L,r>0$. Define next the function
\begin{equation}\label{cosinis}
\phi(s,t):= \cos\left(\frac{\pi s}{2L}\right) \cos\left(\frac{\pi t}{2r}\right).
\end{equation}
Then, $\phi$ satisfies the following properties:

\begin{enumerate}
 \item
$\phi>0$ in the interior of $\Omega_0$, and $\phi=0$ on $\parc \Omega_0$.
 \item
If $L,r$ are large enough, $\cL_g [\phi]> 0$ in the interior of $\Omega_0$.
\end{enumerate}
For the second property, simply note that, by \eqref{opeci} and \eqref{cosinis}, we have $$\cL_g[\phi]= \left(-A\left(\frac{\pi}{2L}\right)^2  - B \left(\frac{\pi}{2r}\right)^2 + C\right)\phi,$$ and that $A,B,C$ are constants with $C>0$.

Let now $S_0(\tau)$ denote the normal variation of the compact domain $\Omega_0\subset S_0$ given by \eqref{vari} with respect to the function $\phi$ in \eqref{cosinis}. Note that, for $L,r$ large enough, the operator $\cW(\tau)$ in \eqref{wt} associated to this variation satisfies $\cW'(0)>0$, by \eqref{operator} and $\cL_g[\phi]>0$. It follows then that for $\tau\in (-\epsilon,\epsilon)$ small enough, we have the following properties when we view the surfaces in the $(s,t,\rho)$ coordinates:

\begin{enumerate}
\item
$S_0(\tau)$ is a compact immersed surface, with boundary $\parc S_0(\tau)=\parc \Omega_0\subset S_0$.
 \item
If $\tau<0$ (resp. $\tau>0$), the interior of $S_0(\tau)$ lies above (resp. below) the plane $S_0$; note that this follows by \eqref{vari}, since $\phi>0$ in the interior of $\Omega_0$ and the unit normal of $S_0$ is vertical and points downwards in the $(s,t,\rho)$-coordinates.
 \item
If $H_{\tau},K_{\tau}$ denote the mean curvature and Gauss curvature of $S_0(\tau)$, and $\tau<0$ (resp. $\tau>0$), then it holds 
 \begin{equation}\label{compa}
H_{\tau} - g(H_{\tau}^2-K_{\tau})<0, \hspace{0.5cm} \text{(resp. $>0$).} 
 \end{equation}
This follows since $\cW(0)=0$ and $\cW'(0)>0$.
\end{enumerate}

We next make a comparison argument between $S_0(\tau)$ and the graph $M$ defined above, with respect to the coordinates $(s,t,\rho)$. See Figure \ref{fig:cili}.

 \begin{figure}[htbp]
    \centering
    \includegraphics[width=.55\textwidth]{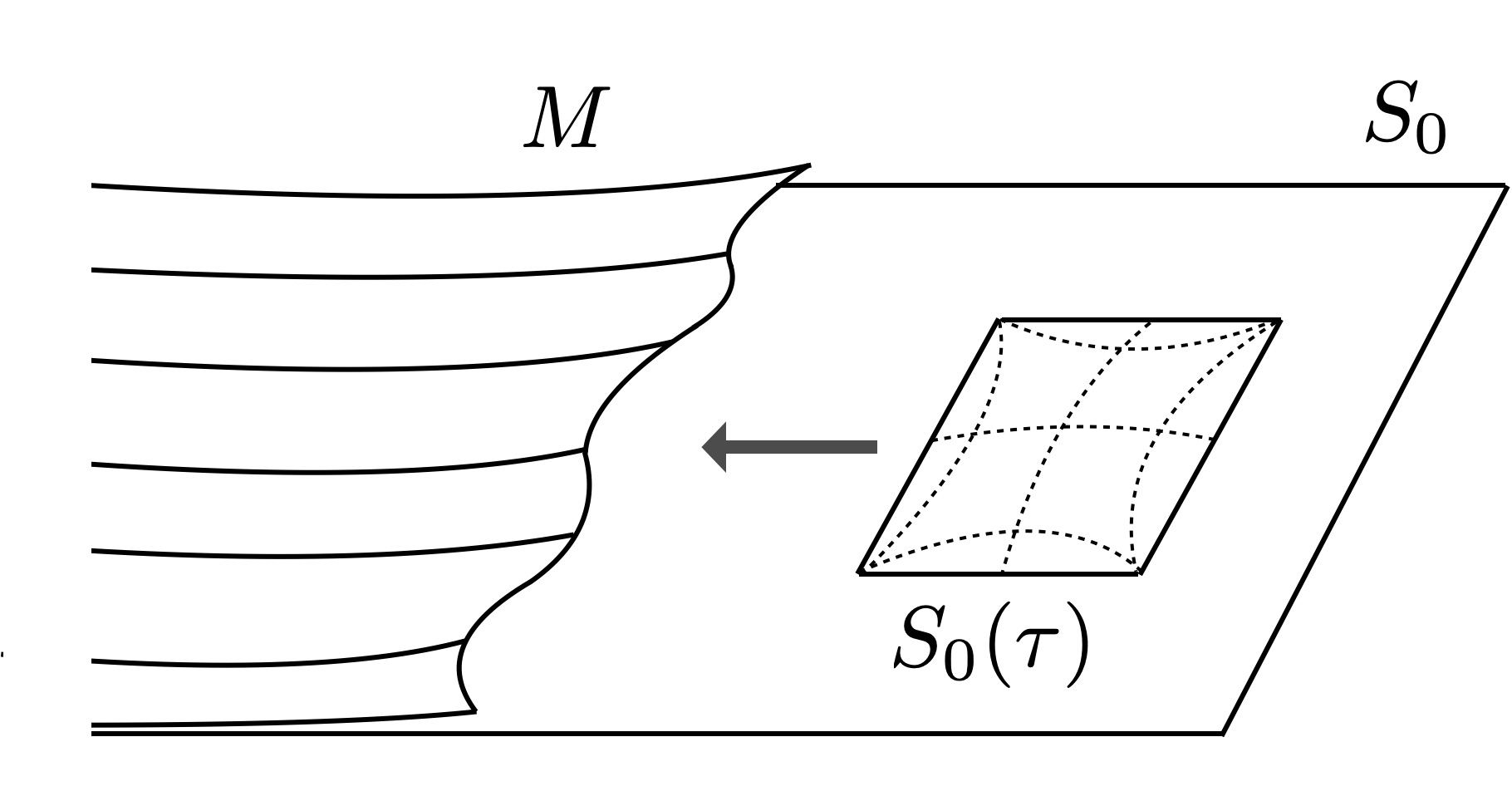}
    \caption{Sliding process of the graph $S_0(\tau)$ over the plane $S_0$, as we make $t$ decrease to $-\8$ from an initial disjoint possition from $M$.}
    \label{fig:cili}
\end{figure}

Note that the boundary $\parc M$ is at a positive distance from the plane $S_0$ when we restrict to the strip of $S_0$ given by $\{(s,t): |s|\leq L\}$. Thus, taking $\tau<0$ sufficiently close to $0$, we may assume that the maximum height of $S_0(\tau)$ over $S_0$ is smaller than this distance, and so all translations of $S_0(\tau)$ in the $t$-direction are disjoint from $\parc M$. Note that both $M$ and the interior of $S_0(\tau)$ lie above $S_0$. Then, we can slide $S_0(\tau)$ horizontally by increasing the $t$-coordinate, until it is disjoint from $M$, and then start sliding it again but in the opposite direction (i.e., making $t$ decrease). Since $M$ converges asymptotically to $S_0$ as $t\to -\8$ and we have avoided $\parc M$ in this sliding process, it is clear that we will eventually find an interior first contact point between $M$ and $S_0(\tau)$. Around this first contact point, $S_0(\tau)$ lies below $M$ in the $(s,t,\rho)$-coordinates. That is, $S_0(\tau)$ lies on the side of $M$ to which their common unit normal points at. But now, observe that $H-g(H^2-K)=0$ on $M$ by \eqref{weq2}, and that $S_0(\tau)$ satisfies \eqref{compa} for $\tau<0$. Since the Weingarten equation \eqref{weq2} is elliptic and $S_0(\tau)$ lies \emph{below} $M$ in the previous sense, this situation contradicts the comparison principle for fully nonlinear elliptic PDEs (see e.g. Theorem 17.1 in \cite{GT}).

This contradiction finishes the proof of Assertion \ref{ass:4}. Let us point out that, in the situation where our initial graph $\cU_0\subset \Sigma$ lies \emph{inside} the cylinder $\Gamma\times\R$, the same argument applies, but now we should take $\tau>0$ so that the surfaces $M$ and $S_0(\tau)$ lie below $S_0$, and contradict again the comparison principle; for this, note the change of sign in \eqref{compa}.
\end{proof}

\section{Bernstein-type theorem in the uniformly elliptic case}\label{sec:4}

In this section we prove a curvature estimate (Theorem \ref{curves2}) that, together with Theorem \ref{b2}, classify the complete, uniformly elliptic Weingarten multigraphs:

\begin{theorem}\label{unifeth}
Planes are the only complete, uniformly elliptic Weingarten multigraphs in $\R^3$.
\end{theorem}
\begin{proof}
Let $\Sigma$ be a complete multigraph that satisfies a uniformly elliptic Weingarten equation. By Theorem \ref{curves2} below, $\Sigma$ has bounded second fundamental form. Thus, $\Sigma$ is a plane, by Theorem \ref{b2}.
\end{proof}

So, it remains to prove the following curvature estimate. 
\begin{theorem}\label{curves2}
Let $\Sigma$ be a complete surface in $\R^3$, possibly with boundary $\parc \Sigma$, and whose Gauss map image $N(\Sigma)$ is contained in an open hemisphere of $\S^2$. Assume that $\Sigma$ satisfies a uniformly elliptic Weingarten equation \eqref{weq2} for some $g:[0,\8)\to \R$, and let $\Lambda>0$ denote the ellipticity constant of $g$ in \eqref{unife2}.

Then, for every $d>0$ there exists a constant $C=C(\Lambda,g(0),d)$ such that for each $p\in \Sigma$ with $d_{\Sigma} (p,\parc \Sigma)\geq d$, it holds $$|\sigma (p)| \leq C.$$ Here, $d_{\Sigma}$ and $|\sigma|$ denote, respectively, the distance function in $\Sigma$ and the norm of the second fundamental form of $\Sigma$.
\end{theorem}
\begin{proof}
The basic strategy of the argument is inspired by a general curvature estimate for stable CMC surfaces in Riemannian $3$-manifolds by Rosenberg, Souam and Toubiana \cite{RST}. For other adaptations of the Rosenberg-Souam-Toubiana estimate to different geometric theories, see \cite{BGM,GMT}. 

To start, arguing by contradiction, assume that there is a sequence of complete immersed surfaces $\psi_n:\Sigma_n\fl\r^3$, possibly with boundary, such that:
 \begin{enumerate}
 \item
The Gauss map image of each $\Sigma_n$ lies in the upper hemisphere $\S_+^2$.
 \item
Each $\Sigma_n$ satisfies a uniformly elliptic Weingarten equation $H=g_n(H^2-K)$, with ellipticity constant $\Lambda$, and $g_n(0)=g(0)$.
 \item
There exist points $p_n\in\Sigma_n$ such that $d_{\Sigma_n}(p_n,\partial\Sigma_n)\geq d$ and $|\sigma_{\Sigma_n}(p_n)|>n$.
 \end{enumerate}

Let us first of all explain the idea behind the proof, in the (known) case of CMC surfaces. First, one makes a blow-up process to the immersions $\psi_n$ after sending the points $p_n$ to the origin, to obtain new immersions $\varphi_n=\landa_n\psi_n$ with $\landa_n\to \8$, such that the second fundamental forms of the $\varphi_n$ are uniformly bounded, and equal to $1$ at the origin. Then, a standard compactness argument of CMC surface theory would prove that a subsequence of the $\varphi_n$ converges uniformly on compact sets to a complete minimal surface $\Sigma_0$ in $\R^3$, that would have Gauss map image contained in a closed hemisphere, and non-zero Gauss curvature at the origin. This would contradict the classical Osserman theorem according to which the Gauss map image of a complete, non-planar minimal surface is dense in $\S^2$, thus giving the desired curvature estimate. 

To prove the above compactness property, a key point is to ensure that the bound of the second fundamental form implies local uniform $C^{2,\alfa}$ estimates for all the immersions $\varphi_n$. In the CMC case, this follows easily by Schauder theory (see Chapter 6 in \cite{GT}), because the CMC equation is quasilinear.

In order to extend these well-known CMC ideas to our general elliptic Weingarten setting, the two main sources of complication are, on the one hand, that the fully nonlinear nature of the Weingarten equation prevents the direct use of Schauder estimates in order to obtain local uniform $C^{2,\alfa}$ estimates for the sequence of surfaces $\varphi_n$; and, on the other hand, that even if the limit surface $\Sigma_0$ exists, it will not be minimal or satisfy an elliptic Weingarten equation (there is no $C^1$ convergence of the equations in this case).

Taking these considerations in mind, we split the proof of Theorem \ref{curves2} into several steps:

\vspace{0.2cm}

\noindent {\bf Step 1:} \emph{A blow-up process}

\vspace{0.2cm}

Let $D_n=D_{\Sigma_n}(p_n,d/2)$ be the compact metric disk in $\Sigma_n$ of center $p_n$ and radius $d/2$, and let $q_n$ be the maximum in $D_n$ of the function
$$
h_n(q)=|\sigma_{\Sigma_n}(q)|d_{\Sigma_n}(q,\partial D_n),\hspace{0.5cm} q\in D_n.
$$

Obviously, $q_n$ lies in the interior of $D_n$ since $h_n$ vanishes on $\partial D_n$. Define next $\lambda_n:=|\sigma_{\Sigma_n}(q_n)|$ and $r_n:=d_{\Sigma_n}(q_n,\partial D_n)$. Then,
\begin{equation}\label{lann}
\lambda_n r_n=|\sigma_{\Sigma_n}(q_n)| \, d_{\Sigma_n}(q_n,\partial D_n)=h_n(q_n)\geq h_n(p_n)>n\ \frac{d}{2}.
\end{equation}
Thus, $\lim _n \lambda_n = \infty$. Also, observe that if we let $\hat{D}_n:=D_{\Sigma_n}(q_n,r_n/2)\subset D_n$, then for any $w_n\in \hat{D}_n$ we have
\begin{equation}\label{destr}
d_{\Sigma_n}(q_n,\partial D_n)\leq 2 d_{\Sigma_n}(w_n,\partial D_n).
\end{equation}
Consider next the immersions $\varphi_n:= \landa_n \psi_n :\hat{D}_n\subset \Sigma_n\fl\r^3$. Then, by
\eqref{destr}, we have for any $w_n\in\hat{D}_n$ that 
\begin{equation}\label{unisec}
|\hat{\sigma}_{n} (w_n)|= \frac{|\sigma_{\Sigma_n}(w_n)|}{\lambda_n} =\frac{h_n(w_n)}{\lambda_n d_{\Sigma_n}(w_n,\partial D_n)}\leq\frac{h_n(q_n)}{\lambda_n d_{\Sigma_n}(w_n,\partial D_n)}\leq 2,
\end{equation}
where $\hat{\sigma}_n$ is the second fundamental form of $\varphi_n$. 
Thus, the norms of the $\hat{\sigma}_n$ are uniformly bounded, and moreover, $|\hat{\sigma}_n(q_n)|=1$. Also, by \eqref{lann}, the radii of the disks $\hat{D}_n$ with respect to the metric induced by $\varphi_n$ diverge to infinity.

Finally, observe that since $\psi_n$ satisfies the Weingarten equation $H=g_n(H^2-K)$, it follows that $\varphi_n$ verifies the corresponding uniformly elliptic Weingarten equation 
 \begin{equation}\label{wefi}
H=\cG_n (H^2-K), \hspace{1cm} \cG_n(t):=\frac{1}{\landa_n} g_n(\landa_n^2 t).
 \end{equation} 
Note that $4t(\cG_n'(t))^2\leq \Lambda<1$, for all $t\in [0,\8)$. That is, the ellipticity constant associated to each $\cG_n$ is also $\Lambda$. It is important to note here that the Weingarten equations \eqref{wefi} do not generally converge $C^1$ to an elliptic Weingarten equation as $\landa_n\to \8$.
\vspace{0.2cm}

\noindent {\bf Step 2:} \emph{A local uniform $C^{2,\alfa}$-estimate for the blown-up immersions}
\vspace{0.2cm}

Assume after a translation of each $\varphi_n$ that $\varphi_n(q_n)=0$ for all $n$. Consider a subsequence of the immersions $\varphi_n$ so that the unit normals at $\varphi_n(q_n)$ converge to some $N_0\in \overline{\S_+^2}$, and choose, after a linear isometry of $\R^3$, new Euclidean coordinates $(x_1,x_2,x_3)$ such that $N_0=(0,0,1)$.

Recall that we have the bound $|\hat{\sigma}_n|\leq 2$ on $\hat{D}_n$, and so we are in the conditions of Remark \ref{2ff}. Then, using this remark and the fact that the unit normals of the $\varphi_n$ converge to $(0,0,1)$ at the origin, it follows that there exist positive constants $\delta_0,\mu_0$ (that correspond to $\delta=\delta(\sigma_0)$, $\mu=\mu(\sigma_0)$ for $\sigma_0=2$ in Remark \ref{2ff}) such that for each $n$ large enough, a neighborhood in $\varphi_n(\hat{D}_n)$ of the origin is given by the graph $x_3=v_n(x_1,x_2)$ of a function $v_n$ defined on the disk $B_{\delta_0}\subset\r^2$ centered at the origin and of radius $\delta_0$, and also:
\begin{enumerate}
\item[(i)] $|Dv_n|<3/2$ in $B_{\delta_0}$.
 \item[(ii)]
$\|v_n\|_{C^2(B_{\delta_0})} \leq \mu_0$. 
\end{enumerate}
Since $\varphi_n$ satisfies \eqref{wefi}, it follows that $v_n(x_1,x_2)$ is a solution to the uniformly elliptic PDE 
 \begin{equation}\label{uen}
F^n (v_{x_1},v_{x_2},v_{x_1x_1},v_{x_1x_2},v_{x_2x_2})=0,
 \end{equation} 
where $F^n(p,q,r,s,t) \in C^{2} (\R^5)$ is given by 
 \begin{equation}\label{uen2}
F^n(p,q,r,s,t)=\cH-\cG_n(\cH^2-\mathcal{K}),
 \end{equation} 
and $\cH, \mathcal{K}$ are defined in \eqref{haka}. Note that, by conditions (i), (ii) above, the images of the sets $(Dv_n(B_{\delta_0}),D^2v_n(B_{\delta_0}))$ lie in the fixed compact set $\Theta$ of $\R^5$ given by 
 \begin{equation}\label{comteta}
\Theta:=\{(p,q,r,s,t) : p^2+q^2\leq 9/4, \, |p|+|q|+|r|+|s|+|t|\leq  \mu_0\}.
\end{equation}

In order to ensure convergence of the immersions $\varphi_n$, we will prove that there exists a uniform bound of the $C^{2,\alfa}$ norm of $v_n$ in $B_{\delta'}$, for some fixed $\delta'\in (0,\delta_0)$, some $\alfa\in (0,1)$, and for all $n$. In order to do this, we will use Nirenberg's a priori estimate for fully nonlinear elliptic equations in dimension two (\cite[Theorem I]{Nir}), applied to each elliptic equation \eqref{uen}. To apply Nirenberg's theorem, it suffices to check the following two conditions for the compact set $\Theta$ in \eqref{comteta}:

\begin{enumerate}
\item[(a)]
All first derivatives of all $F^n$ are uniformly bounded in $\Theta$.
 \item[(b)]
There exists a constant $\landa>0$ such that 
 \begin{equation}\label{prob}
 F^n_r \, \xi^2+ F^n_s \, \xi \eta + F^n_t \, \eta^2 \geq \landa (\xi^2+\eta^2)
  \end{equation} 
at every point of $\Theta$, for any $(\xi,\eta)\in \R^2$ and any $n$.\end{enumerate}

Let us prove these two conditions.
The expression $\cH^2-\mathcal{K}$ is clearly homogeneous and quadratic in $(r,s,t)$, for each $(p,q)$ fixed.
A computation shows that it has one zero eigenvalue, and two positive eigenvalues $\landa_1^2,\landa_2^2$ given by
 \begin{equation}\label{landai}
\landa_i^2= \landa_i^2 (p,q)= \frac{6 + p^4 + 6 q^2 + q^4 + p^2 (6 + 4 q^2) \pm \sqrt{
  \mathcal{Q}_4(p^2,q^2)}}{8(1+p^2+q^2)}>0,
  \end{equation} 
where $\mathcal{Q}_4(x,y)$ is the polynomial of degree $4$ $$\mathcal{Q}_4(x,y)=x^4 + y^3 (8 y-4) + 
    2 x^2 y (14 + 9 y) + (y^2-2y -2)^2 + 
    4 x (2 + 10 y + 7 y^2 + 2 y^3).$$
Moreover, it is easy to check from \eqref{landai} that both $\landa_i(p,q)$ are bounded from below by a positive constant when we restrict to the compact set $\Theta$. So, after an orthogonal change of coordinates $(r,s,t)\mapsto (\bar{r},\bar{s},\bar{t})$, where the related orthogonal matrix depends on $(p,q)$, we can write 
 \begin{equation}\label{acheka}
(\cH^2-\mathcal{K})(p,q,r,s,t)= \landa_1^2 \, \bar{r}^2 + \landa_2^2\, \bar{t}^2,
 \end{equation} 
where here $\bar{r},\bar{t}$ depend on $(p,q,r,s,t)$, the dependence on $(r,s,t)$ being linear. 

All these functions $\landa_i,\bar{r},\bar{s},\bar{t}$ can be chosen to be real analytic in their arguments, except around the points $(p,q,r,s,t)$ where $\landa_1(p,q)=\landa_2(p,q)$, i.e., around the points where the eigenvalue multiplicity changes. Call $\cB$ to this set of points. 

We claim that $\cB\cap \Theta$ is empty. To see this, first observe that, by \eqref{landai}, $\cB$ is given by the expression $\mathcal{Q}_4(p^2,q^2)=0$. We can rewrite $\mathcal{Q}_4$ as $$\mathcal{Q}_4(x,y)=((x + y)^2 - 2 (x + y) - 2)^2 +4 x y (10 + x^2 + 10 y + y^2 + x (10 + 3 y)).$$ So, $$\mathcal{Q}_4(p^2,q^2)\geq \big((p^2 + q^2)^2 - 2 (p^2 + q^2) - 2\big)^2,$$ and the expression in the right hand side vanishes only when $p^2+q^2=1+\sqrt{3}$. By the definition of $\Theta$ in \eqref{comteta}, it is clear then that $\mathcal{Q}_4 (p^2,q^2)>0$ in $\Theta$, since $p^2+q^2\leq 9/4$ in $\Theta$. Thus, $\Theta$ does not intersect $\cB$. In particular, the functions $\landa_i,\bar{r},\bar{s},\bar{t}$ are real analytic in $\Theta$. Now, note that for any $w\in \{p,q,r,s,t\}$ we have in $\Theta$, by \eqref{acheka},

\begin{equation}\label{ar1}
\def\arraystretch{2.5}\begin{array}{ccl} 
\displaystyle\left| \frac{(\cH^2-\mathcal{K})_w}{\sqrt{\cH^2-\mathcal{K}}}\right| & =& \displaystyle \frac{| (\landa_1^2)_w \,\bar{r}^2 + (\landa_2^2)_w\, \bar{t}^2 + 2 \landa_1^2 \,\bar{r} \,\bar{r}_w + 2 \landa_2^2\, \bar{t} \bar{t}_w|}{\sqrt{\landa_1^2\, \bar{r}^2 + \landa_2^2\, \bar{t}^2}} \\ & \leq & \displaystyle \frac{|(\landa_1^2)_w \, \bar{r}|}{\landa_1} + \frac{|(\landa_2^2)_w \, \bar{t}|}{\landa_2} + |2 \landa_1\, \bar{r}_w| + |2 \landa_2 \, \bar{t}_w | \\ & \leq & C_1=C_1(\Theta)
\end{array}
\end{equation} for some positive constant $C_1$ depending on $\Theta$. From here and \eqref{uen2}, we have in $\Theta$:
\begin{equation}\label{ar2}
\def\arraystretch{2.5}\begin{array}{ccl} 
|F^n_w| & =& \left|\cH_w - \sqrt{\cH^2-\mathcal{K}}\,\cG_n'(\cH^2-\mathcal{K}) \, \displaystyle\frac{(\cH^2-\mathcal{K})_w}{\sqrt{\cH^2-\mathcal{K}}}\right|,\\
& \leq & {\rm max}_{\Theta} |\cH_w| + \frac{1}{2} C_1(\Theta) \leq C_2(\Theta)\end{array}
\end{equation} 
where we have used that $\sqrt{t}|\cG_n'(t)|< 1/2$ by the ellipticity condition on $\cG(t)$. Thus, all the first derivatives of $F^n$ with respect to any $w\in \{p,q,r,s,t\}$ are uniformly bounded in $\Theta$. This proves property (a).

Once we know that (a) holds, the proof of (b) is a straightforward consequence of the fact that all the equations \eqref{uen} are uniformly elliptic on $\Theta$ for the same ellipticity constant, since all the $\cG_n$ satisfy the uniform condition $4t(\cG_n'(t))^2\leq \Lambda<1$ for the same $\Lambda$. Thus, \eqref{prob} holds for some $\landa=\landa(\Lambda,\Theta)$.

With this, we are in the conditions of Theorem I in \cite{Nir} (alternatively, see also Theorem 17.9 in \cite{GT}), which implies what follows in our situation. Fix $\delta'\in (0,\delta_0)$ once and for all, from now on. Then, there exist constants $C'>0$ and $\alfa\in (0,1)$ such that 
 \begin{equation}\label{holdest}
||v_n||_{C^{2,\alfa}(B_{\delta'})} \leq C'
 \end{equation}
 for all $n$. Here $C',\alfa$ depend only on $\Lambda$, in the following sense: at first, these constants depend on $\delta_0,\mu_0,\delta'$, the ellipticity constant $\Lambda$ and the bounds on the derivatives of $F^n$ in $\Theta$. Nonetheless, $\delta_0,\mu_0$ are determined by the condition that $|\hat{\sigma}_{n}|\leq 2$, and the bounds \eqref{ar2} obtained for $F^n_w$ on $\Theta$ are independent of the equation $F^n$, i.e., they only depend on $\delta_0$. Since $\delta'$ has been considered fixed, the numbers $C',\alfa$ only depend on $\Lambda$.

\vspace{0.2cm}

\noindent {\bf Step 3:} \emph{Existence and properties of a limit surface of the blown-up immersions}

\vspace{0.1cm}

It follows by the estimate \eqref{holdest} that the set $\{v_n\}_n$ is bounded in the $C^{2,\alfa}(B_{\delta'})$-norm, and therefore is precompact in the $C^{2,\beta}(B_{\delta'})$-norm, for any $\beta\in (0,\alfa)$. Thus, a subsequence of the $v_n$ converges uniformly in the $C^{2,\beta}(B_{\delta'})$-norm to some function $v^0 \in C^{2,\beta} (B_{\delta'})$; here, $\beta$ is any number in $(0,\alfa)$, that we also consider fixed from now on.

Once here, we can apply a typical diagonal extension process and deduce that the graph $x_3=v^0(x_1,x_2)$ can be extended to a complete immersion $\psi^0:\Sigma_0\flecha \R^3$ that, by construction, is a limit in the $C^2$-topology on compact sets of a subsequence of the immersions $\varphi_n$. We denote this limit surface simply by $\Sigma_0$. That $\Sigma_0$ is complete follows since the radii of the $\hat{D}_n$ go to $\8$.

Note that $\Sigma_0$ is not, in general, an elliptic Weingarten surface since, as explained before, the elliptic Weingarten equations \eqref{wefi} do not necessarily converge $C^1$ to a Weingarten equation.

We single out the following list of properties of $\Sigma_0$, that will be proved below.

\begin{enumerate}
\item[(P1)]
$\Sigma_0$ is complete.
 \item[(P2)]
$\Sigma_0$ has bounded second fundamental form. 
 \item[(P3)]
The Gauss map image $N(\Sigma_0)$ lies in the closed hemisphere $\overline{\S_+^2}$.
\item[(P4)]
The Gauss map $N:\Sigma_0\flecha \S^2$ is quasiconformal.
 \item[(P5)]
$\Sigma_0$ is not a plane.
\end{enumerate}

The fact that $\Sigma_0$ is a complete surface was explained above. The second fundamental form of $\Sigma_0$ is bounded since it is a $C^2$-limit of the immersions $\varphi_n(\hat{D}_n)$, and $|\hat{\sigma}_n|\leq 2$ on $\hat{D}_n$. That $N(\Sigma_0)$ lies in $\overline{\S_+^2}$ is also immediate, since all the $\varphi_n$ are multigraphs (here $\S_+^2$ denotes the upper hemisphere in the original $(x,y,z)$-coordinates of $\R^3$). 
Since the norm of the second fundamental form of $\varphi_n(\hat{D}_n)$ is equal to $1$ at the origin for all $n$, the same happens to $\Sigma_0$; thus, $\Sigma_0$ is not a plane.

 \begin{figure}[htbp]
    \centering
        \includegraphics[width=.5\textwidth]{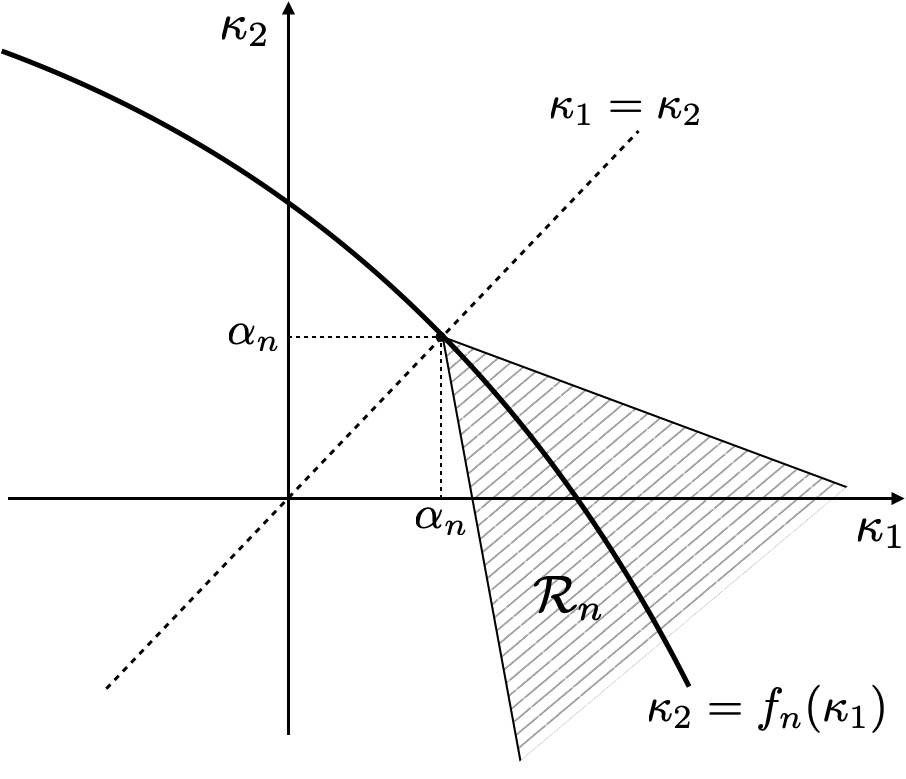}
    \caption{The region $\mathcal{R}_n$.}
    \label{fig:rn}
\end{figure}

So, the only property that remains to check is (P4), i.e., that $\Sigma_0$ has quasiconformal Gauss map. To start, let us rewrite the uniformly elliptic Weingarten equation \eqref{wefi} satisfied by $\varphi_n$ in the form \eqref{wein2}; that is, we rewrite \eqref{wefi} as $\kappa_2=f_n(\kappa_1),$ where $f_n\in C^{2}(\R)$ satisfies $f_n\circ f_n={\rm Id}$ and the uniform ellipticity condition \eqref{unife}. Let $\kappa_1^n\geq \kappa_2^n$ denote the principal curvatures of $\varphi_n$. Then, by the bounds in \eqref{unife}, it is clear that there exist $m_1,m_2<0$ (independent of $n$) such that, for each $n$, the curvature diagram $$(\kappa_1^n(\hat{D}_n),\kappa_2^n (\hat{D}_n))\subset \R^2$$ lies in the wedge region of the plane $$\cR_n:= \{(x,y): x\geq y, m_1 (x-\alfa_n)\leq y-\alfa_n\leq m_2 (x-\alfa_n)\}\subset \R^2,$$ where $\alfa_n$ is the umbilical constant of \eqref{wefi}, given by $\cG_n(0)=\alfa_n$, or equivalently by $f_n(\alfa_n)=\alfa_n$. See Figure \ref{fig:rn}. Note that $\alfa_n=g(0)/\landa_n\to 0$ as $n\to \8$. Thus, the regions $\cR_n$ converge to the region $\cR$ in \eqref{wedge}, and it follows that the (bounded) set $(\kappa_1(\Sigma_0),\kappa_2(\Sigma_0))\subset \R^2$ lies inside this wedge region $\cR$, where $\kappa_1\geq \kappa_2$ are the principal curvatures of $\Sigma_0$. By the arguments explained after Definition \ref{def:cuasic}, we deduce that the Gauss map of $\Sigma_0$ is quasiconformal, as claimed.

\vspace{0.2cm}

\noindent {\bf Step 4:} \emph{A surface $\Sigma_0$ with the properties (P1)-(P5) of Step 3 cannot exist.}

\vspace{0.2cm}

This is immediate, by Theorem \ref{alfa}. Hence, we obtain a contradiction, which completes the proof of Theorem \ref{curves2}.
\end{proof}

\section{A Bernstein-type theorem in the non-uniformly elliptic case}\label{sec:5}

Let $\Sigma$ be a complete, non-compact surface in $\R^3$, with principal curvatures $\kappa_1\geq \kappa_2$. It follows then from Section \ref{sec:cuasi} that $\Sigma$ has quasiconformal Gauss map if and only if its curvature diagram $(\kappa_1(\Sigma),\kappa_2(\Sigma))$ is contained in a \emph{wedge} region $\cR$ of $\R^2$ that lies between two straight half-lines with negative slopes that pass through the origin, as in Figure \ref{fig:wedges1}, left. See Remark \ref{caposi} and the discussion after Definition \ref{def:cuasic}.

Motivated by this, we study next a different curvature diagram restriction.

Let $\varphi_1\leq \varphi_2:[0,\8)\flecha \R$ be two decreasing $C^1$ functions with $\varphi_i(0)=0$, $\varphi_i'(0)=m_i<0$, and assume that both $\varphi_i$ are bounded from below, i.e., $s_0\leq \varphi_1\leq \varphi_2\leq 0$ for some $s_0<0$. Consider the planar regions $\cR_{\varphi}$, $\cR_{\varphi}^*$ given by (see Figure \ref{fig:hiperbolas})
\begin{equation}\label{errefi}
\cR_{\varphi}:=\{(x,y): x\geq 0, \varphi_1(x)\leq y \leq \varphi_2(x)\}, \hspace{0.5cm} \cR_{\varphi}^*:=\{(-y,-x) : (x,y)\in \cR_{\varphi}\}.
\end{equation}

\begin{figure}[htbp]
    \centering
    \includegraphics[width=1\textwidth]{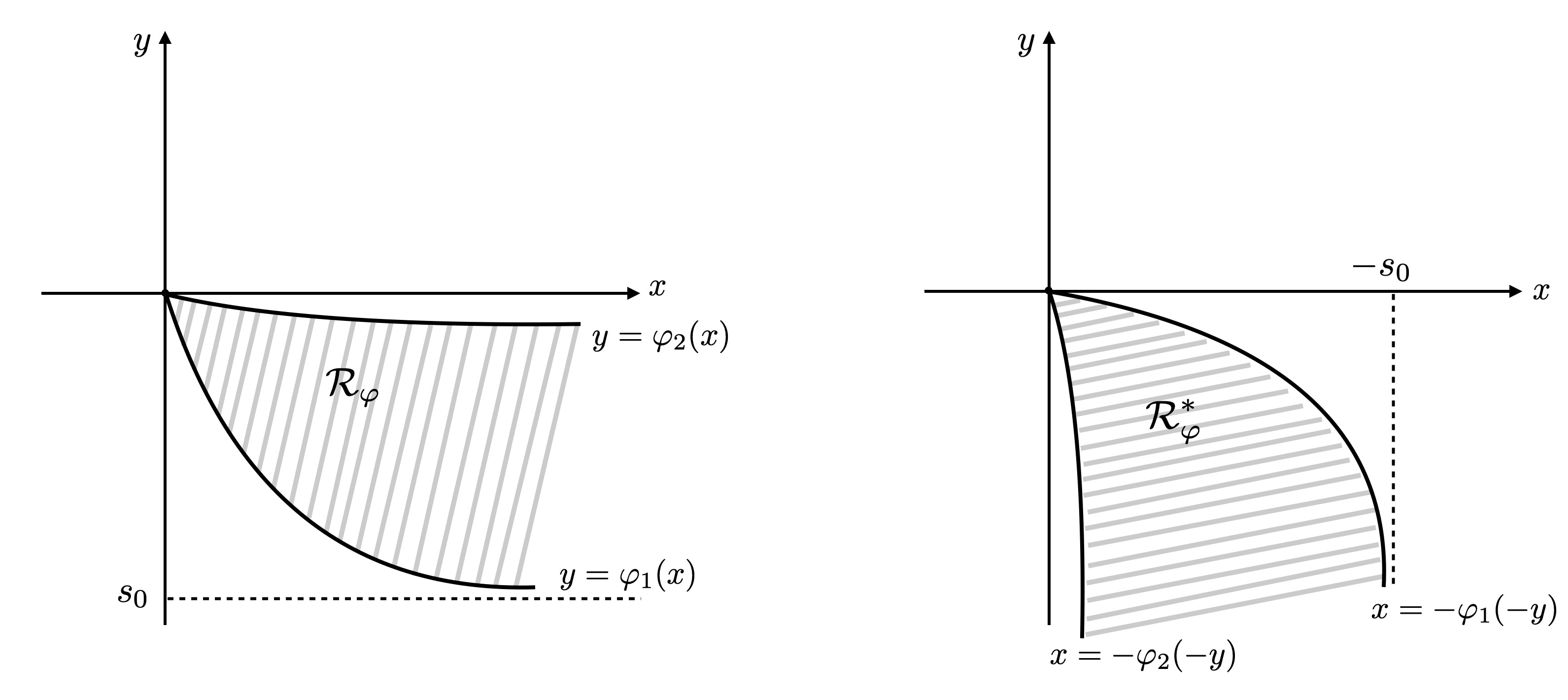}
    \caption{The regions $\cR_{\varphi}$ and $\cR_{\varphi}^*$.}
    \label{fig:hiperbolas}
\end{figure}

\begin{theorem}\label{th:erre}
Let $\Sigma$ be a complete multigraph whose curvature diagram $(\kappa_1(\Sigma),\kappa_2(\Sigma))$ is contained in a planar region of the form $\cR_{\varphi}$ or $\cR_{\varphi}^*$. Then, $\Sigma$ is a plane.
\end{theorem}
\begin{proof}
We will prove the result for $\cR_{\varphi}$; the result for $\cR_{\varphi}^*$ follows then by a change of orientation. Since the curvature diagram of $\Sigma$ lies in some region $\cR_{\varphi}$, if we take $t_0<s_0$, we see that there exists $\ep>0$ such that 
 \begin{equation}\label{condicur}
|\kappa_i(p)-t_0|\geq \ep, \hspace{0.5cm} \forall p\in \Sigma, \hspace{0.5cm} i=1,2.
\end{equation} 

Write $\psi:\Sigma\flecha \R^3$ for the immersion of $\Sigma$ into $\R^3$, and consider for $a:=1/t_0$ the \emph{parallel surface} of $\Sigma$ at a distance $a$, given by $\psi^a := \psi + a N :\Sigma\flecha \R^3$, where $N:\Sigma\flecha \S^2$ is the Gauss map of $\Sigma$. In general, $\psi^a$ may have singular points; indeed, its induced metric $g^a$ can be expressed at any point in terms of an orthonormal basis of principal directions $\{e_1,e_2\}$ of $\Sigma$ as 
$$g^a(e_i,e_j)=(1-a \kappa_i)^2 \delta_{ij},$$ where $\kappa_i$ is the principal curvature of $\Sigma$ in the direction $e_i$. However, in our present situation, the condition \eqref{condicur} ensures that $\psi^a$ is everywhere regular. Moreover, it also follows from \eqref{condicur} and the expression of $g^a$ that $g^a(e_i,e_j)\geq a^2 \ep^2 \delta_{ij}$, and so $\psi^a$ is a complete surface.

In addition, the Gauss map of $\psi^a$ is equal to $N$ (thus $\psi^a$ is also a multigraph), and $\{e_1,e_2\}$ are also principal directions for $\psi^a$. The principal curvatures of $\psi^a$ are given then by 
 \begin{equation}\label{weipa}
 \kappa_i^a= \frac{\kappa_i}{1-a \kappa_i}, \hspace{1cm} i=1,2.
  \end{equation}
From this expression and \eqref{condicur}, it is clear that the $\kappa_i^a$ are uniformly bounded, i.e., that $\psi^a$ has bounded second fundamental form.

We check next that $\psi^a$ has quasiconformal Gauss map. Let $F_a(t):= \frac{t}{1-at}$. Note that $F_a'>0$ and that $(F_a)^{-1}=F_{-a}$. Also, by \eqref{weipa}, we have $\kappa_i^a= F_a (\kappa_i)$. In particular, $\kappa_1^a\geq \kappa_2^a$. Since the curvature diagram of $\Sigma$ lies in $\cR_{\varphi}$, we have $\varphi_1(\kappa_1)\leq \kappa_2 \leq \varphi_2(\kappa_1)$. Thus, since $F_a$ is increasing, we obtain
\begin{equation}\label{paracu}F_a \circ \varphi_1 \circ F_{-a} (\kappa_1^a) \leq \kappa_2^a \leq F_a \circ \varphi_2 \circ F_{-a} (\kappa_1^a).
\end{equation} 
The functions $F_a\circ \varphi_i \circ F_{-a}$ are strictly decreasing, with negative derivative at the origin. Hence, since both $\kappa_i^a$ are bounded, this implies that $(\kappa_1^a(\Sigma),\kappa_2^a(\Sigma))$ lies in a wedge region of $\R^2$ of the form \eqref{wedge}. Thus, $\psi^a$ has quasiconformal Gauss map.

To sum up, $\psi^a$ is a complete multigraph with quasiconformal Gauss map and bounded second fundamental form. By Theorem \ref{alfa}, it is a plane. Hence, $\Sigma$ must also be a plane.
\end{proof}

Theorem \ref{th:erre} has a direct important consequence for elliptic Weingarten surfaces. For this, recall that if we write an elliptic Weingarten equation as \eqref{wein2}, the notation $I_f$ indicates the domain of the function $f$, which is an interval of $\R$. Then, Theorem \ref{bernstein} below follows directly from Theorem \ref{th:erre}, and solves the Bernstein problem for elliptic Weingarten graphs (and more generally for multigraphs with $f(0)=0$) in the case that $I_f\neq \R$.

\begin{theorem}\label{bernstein}
Let $\Sigma$ be a complete multigraph in $\R^3$ that satisfies an elliptic Weingarten equation\eqref{wein2}, with $I_f\neq \R$ and $f(0)=0$. Then $\Sigma$ is a plane.
\end{theorem}
\begin{proof}
First, note that the curvature diagram $(\kappa_1(\Sigma),\kappa_2(\Sigma))$, $\kappa_1\geq \kappa_2$, of $\Sigma$ is contained in the region of $\R^2$ of the form $\{(x,y) : x\geq y, y=f(x)\}$. Note that  $f'<0$ everywhere, with $f(0)=0$. Since $I_f\neq \R$, then $I_f=(a,\8)$ for $a<0$, or $I_f=(-\8,b)$ for $b>0$; see Section \ref{sec:21}. By the symmetry condition $f\circ f ={\rm Id}$, it follows then that $f(x)\to a$ when $x\to \8$ in the first case, and that $f(x)\to -\8$ when $x\to b^-$ in the second case. Thus, the curvature diagram of $\Sigma$ lies in a region of the form $\cR_{\varphi}$ in the first case, and in one of the form $\cR_{\varphi}^*$ in the second one. By Theorem \ref{th:erre}, we conclude then that $\Sigma$ must be a plane. \end{proof}

For the sake of completeness, we reformulate Theorem \ref{bernstein} for the situation in which the Weingarten equation is written as \eqref{weq2}, instead of  \eqref{wein2}: 

\begin{theorem}\label{bernstein2}
Let $g\in C^{2}([0,\8))$ satisfy $g(0)=0$ and:
\begin{enumerate}
\item
$4t (g'(t))^2<1$ for all $t$ (ellipticity condition).
 \item
Either $t+g(t^2)$ or $t-g(t^2)$ is bounded in $[0,\8)$.
\end{enumerate}
Then, any complete multigraph (in particular, any entire graph) in $\R^3$ that satisfies the Weingarten equation $H=g(H^2-K)$ is a plane.
\end{theorem}
\begin{proof}
By Theorem \ref{bernstein}, we only need to check that the second condition on $g$ in the statement is equivalent to the fact that $I_f\neq \R$, for the function $f$ appearing when we rewrite \eqref{weq2} as \eqref{wein2}. Denote $t:=H^2-K$, and note that $\{\kappa_1,\kappa_2\}=g(t) \pm \sqrt{t}$, because of \eqref{weq2}. Therefore, the graph of $f$, given by $\{(x,f(x)) : x\in I_f\}$, is equal to the union 
\begin{equation}\label{eqra} \left\{\left(g(t)-\sqrt{t},g(t)+\sqrt{t}\right) : t\geq 0\right\} \cup \left\{\left(g(t)+\sqrt{t},g(t)-\sqrt{t}\right) : t\geq 0\right\},
 \end{equation}
due to the symmetry condition $f\circ f={\rm Id}$. 
From the ellipticity condition (1), we see that $g(t)+\sqrt{t}$ is strictly increasing (thus, bounded from below), and $g(t)-\sqrt{t}$ is strictly decreasing (thus, bounded from above). Therefore, from \eqref{eqra}, $I_f$ is bounded from below if and only if $g(t)-\sqrt{t}$ is bounded in $[0,\8)$, and $I_f$ is bounded from above if and only if $g(t)+\sqrt{t}$ is bounded in $[0,\8)$. This gives the equivalence of $I_f\neq \R$ with the second condition above.\end{proof}

Conditions (1)-(2) in Theorem \ref{bernstein2} have also appeared in previous works by Sa Earp and Toubiana \cite{ST,ST2} in connection with the existence of \emph{catenoids} and half-space theorems for elliptic Weingarten surfaces of minimal type. See also \cite{EM}.

An \emph{elliptic linear Weingarten surface} is one that satisfies the equation 
 \begin{equation}\label{lw}
2\alfa H + \beta K =\delta, \hspace{1cm}\alfa,\beta,\delta \in \R,
 \end{equation} 
where the ellipticity condition is $\alfa^2+\beta \delta>0$. This family contains surfaces of constant mean curvature ($\beta=0$) and of constant positive curvature ($\alfa=0$), and corresponds to the family of parallel surfaces of the class of CMC surfaces in $\R^3$. However, as the parallel surface procedure usually creates singularities, their global geometry is not equivalent to the class of CMC surfaces.

In terms of $\kappa_1,\kappa_2$, equation \eqref{lw} is written as $\kappa_2=f(\kappa_1)$, where 
 \begin{equation}\label{lew}
f(x)= \frac{\delta-\alfa x}{\alfa +\beta x}.
\end{equation} Note that $I_f\neq \R$ for $f$ as in \eqref{lew} unless $\beta=0$. With this, we have:

\begin{corollary}\label{lwm}
Planes and cylinders are the only complete, elliptic linear Weingarten surfaces in $\R^3$ whose Gauss map image lies in a closed hemisphere of $\S^2$.
\end{corollary}
\begin{proof}
If $\beta=0$, this is the classical theorem of Hoffman, Osserman and Schoen for CMC surfaces, see \cite{HOS}; note that it also follows from Theorem \ref{unifeth} and Lemma \ref{numulti}. If $\beta\neq 0$, then $I_f\neq \R$ and so we can consider $t_0\neq 0$ not in $\overline{I_f}$. Assume that the surface is a multigraph. Then, its parallel surface at a distance $1/t_0$ is a complete multigraph with bounded second fundamental form (see the proof of Theorem \ref{th:erre}) and also an elliptic linear Weingarten surface, by an elementary computation using \eqref{weipa} and \eqref{lew}. Thus, the result follows from Theorem \ref{b2}. Finally, by Lemma \ref{numulti}, we see that if the surface is not a multigraph, it must be a cylinder (note that planes are multigraphs). This completes the proof.
\end{proof}

\def\refname{References}

\vskip 0.2cm

\noindent Isabel Fernández

\noindent Departamento de Matemática Aplicada I,\\ Instituto de Matemáticas IMUS \\ Universidad de Sevilla (Spain).

\noindent  e-mail: {\tt isafer@us.es}

\vskip 0.2cm

\noindent José A. Gálvez

\noindent Departamento de Geometría y Topología,\\ Universidad de Granada (Spain).

\noindent  e-mail: {\tt jagalvez@ugr.es}

\vskip 0.2cm

\noindent Pablo Mira

\noindent Departamento de Matemática Aplicada y Estadística,\\ Universidad Politécnica de Cartagena (Spain).

\noindent  e-mail: {\tt pablo.mira@upct.es}

\vskip 0.4cm

\noindent Research partially supported by MINECO/FEDER Grant no. MTM2016-80313-P

\end{document}